\pdfoutput=1
\documentclass[english]{jnsao} 
\usepackage[utf8]{inputenc}

\usepackage[capitalise,nameinlink]{cleveref}
\numberwithin{equation}{section}
\usepackage{enumitem}
\newlist{assumption}{enumerate}{1}
\setlist[assumption]{label=(\textsc{a}\arabic*)}
\crefname{assumptioni}{Assumption}{Assumptions}

\newcommand{\R}{\mathbb{R}}
\newcommand{\N}{\mathbb{N}}
\newcommand{\1}{\mathbb{1}}
\newcommand{\Linop}{\mathbb{L}}
\newcommand{\card}{\text{card}}

\let\div\relax
\DeclareMathOperator{\div}{\mathrm{div}}

\newcommand{\dx}{\,\mathrm{d}x}
\newcommand{\ds}{\,\mathrm{d}s}
\newcommand{\barK}{K}

\newcommand{\norm}[1]{\|#1\|}
\def\weakto{\rightharpoonup}
\renewcommand{\epsilon}{\varepsilon}

\def\thetitle{No-gap second-order optimality conditions for optimal control of a non-smooth quasilinear elliptic equation}
\shortauthor{Clason, Nhu, Rösch}

\manuscripteprinttype{arxiv}
\manuscripteprint{2003.11478v2}
\date{2020-09-09}
\manuscriptlicense{}
\manuscriptcopyright{}

\title{\thetitle}
\author{Christian Clason\thanks{Faculty of Mathematics, University of Duisburg-Essen, Thea-Leymann-Strasse 9, 45127 Essen, Germany (\email{christian.clason@uni-due.de}, \orcid{0000-0002-9948-842})}
    \and Vu Huu Nhu\thanks{Faculty of Mathematics, University of Duisburg-Essen, Thea-Leymann-Strasse 9, 45127 Essen, Germany; \emph{currently} Department of Scientific Fundamentals, Posts and Telecommunications Institute of Technology, Hanoi, Vietnam (\email{huu.vu@uni-due.de}, \orcid{0000-0003-4279-3937})}
    \and Arnd Rösch\thanks{Faculty of Mathematics, University of Duisburg-Essen, Thea-Leymann-Strasse 9, 45127 Essen, Germany (\email{arnd.roesch@uni-due.de})}
}

\AtBeginDocument{
    \hypersetup{
        pdftitle = {\thetitle},
        pdfauthor = {C.~Clason and V.\,H.~Nhu and A.~Rösch}
    }
}
\acknowledgments{This work was supported by the DFG under the grants CL 487/2-1 and RO
    2462/6-1, both within the priority programme SPP 1962 ``Non-smooth and
    Complementarity-based Distributed Parameter Systems:
    Simulation and Hierarchical Optimization''.
}

\begin{document}
\maketitle

\begin{abstract}
    This paper deals with second-order optimality conditions for a quasilinear elliptic control problem with a nonlinear coefficient in the principal part that is finitely $PC^2$ (continuous and $C^2$ apart from finitely many points). We prove that the control-to-state operator is continuously differentiable even though the nonlinear coefficient is non-smooth. This enables us to establish ``no-gap'' second-order necessary and sufficient optimality conditions in terms of an abstract curvature functional, i.e., for which the sufficient condition only differs from the necessary one in the fact that the inequality is strict. A condition that is equivalent to the second-order sufficient optimality condition and could be useful for error estimates in, e.g., finite element discretizations is also provided.
\end{abstract}

\section{Introduction}
This work is concerned with the quasilinear elliptic optimal control problem
\begin{equation}
    \label{eq:P}
    \tag{P}
    \left\{
        \begin{aligned}
            \min_{u\in L^\infty(\Omega), y\in H^1_0(\Omega)} &J(y,u) := G(y) + \frac{\nu}{2} \norm{u}_{L^2(\Omega)}^2 \\
            \text{s.t.} \quad &-\div [(b + a(y)) \nabla y] = u \, \text{ in } \Omega, \quad y = 0 \, \text{ on } \partial\Omega, \\
            & \alpha(x) \leq u(x) \leq \beta(x) \quad \text{a.e. } x \in \Omega,
        \end{aligned}
    \right.
\end{equation}
with a $C^2$-functional $G:H^1_0(\Omega) \to\R$ for a bounded convex domain $\Omega\subset \R^N$, $N \in \{2,3\}$, a Lipschitz continuous function $b: \overline \Omega \to \R$, a continuous and piecewise twice differentiable function $a: \R \to \R$, functions $\alpha, \beta \in L^\infty(\Omega)$ satisfying $\beta(x) - \alpha(x) \geq \gamma$ for some $\gamma>0$ and almost every $x \in \Omega$, and a positive constant $\nu$. For the precise assumptions on the data of \eqref{eq:P}, we refer to \cref{sec:assumption}.

The state equation in the optimal control problem \eqref{eq:P} arises, for instance, in models of heat conduction with a nonlinear dependence on the temperature $y$ that allows for different behavior in different temperature regimes with sharp phase transitions. Such situations occur, e.g., in the context of steel production, where the thermal conductivity does not change spatially but rather depends on the temperature;  cf.~\cite{Binder1995,Engl1988,Kugler2003}. When the conductivity coefficient is of class $C^2$ in the state variable $y$, second-order necessary and sufficient conditions for such optimal control problems were already obtained in \cite{CasasTroltzsch2009,CasasDhamo2011-2ndOSs}. However, if the coefficient is non-smooth, the standard tools for smooth problems are typically inapplicable, making the analytical and numerical treatment challenging.
The goal of this paper is therefore to derive second-order necessary and sufficient optimality conditions for \eqref{eq:P}. 
Specifically, we introduce a curvature functional in terms of the jumps of the first-order derivatives of $a$ in critical points (see \cref{sec:2nd-OC:estimates}).
Using a second-order Taylor-type expansion and estimates of the functional in terms of the jumps of $a'$ in the optimal state, we derive second-order necessary as well as sufficient conditions that are ``no-gap'' conditions in the sense that the only difference between necessary and sufficient conditions are in the fact that the inequality in the latter are strict (see \cref{thm:2nd-OS-nec,thm:2nd-OS-suf}). 
In addition, we derive an equivalent formulation of the second-order sufficient condition useful for proving error estimates for finite element discretizations of \eqref{eq:P}, which will be the focus of a follow-up work.

\bigskip

Let us comment on related work.
As far as second-order sufficient optimality conditions (SSC) are concerned, there is a rich literature on SSC for smooth PDE constrained optimal control problems; see, e.g., the articles \cite{CasasMateos2002,CasasDhamo2011-2ndOSs,Casas2008,KrumbiegelNeitzelRosch2010,RoschTroltzsch2006,RoschTroltzsch2003}, the seminal book \cite{Troltzsch2010}, the survey \cite{CasasTroltzsch:2015}, as well as the references therein. 

Regarding second-order necessary optimality conditions (SNC) for optimization problems, it is well-known that the second-order derivative of a Lagrangian function (or of the reduced cost functional) is, in general, not less than a so-called ``sigma-term'' \cite[Chap.~3]{BonnansShapiro2000}. This term is defined as the value of a support functional of a second-order tangent set and contributes prominently in the gap between SNC and SSC. If, in addition, the optimization problem satisfies the \emph{polyhedricity} condition, then the ``sigma-term'' and hence the gap vanishes; see, e.g., \cite[Prop.~3.53]{BonnansShapiro2000}. For related works for smooth semilinear PDE-constrained problems, we refer to \cite{BayenBonnansSilva2014,KienNhuSon2017,BonnansZidani1999,BayenSilva2016,Bonnans1998} as well as the references therein for $C^2$ coefficients, while \cite{Tuan2019} treats the case where the nonlinearities are of class $C^1$, but not $C^2$, and second-order sequentially directionally differentiable.
In particular, \cite{Bonnans1998} is to the best of our knowledge the first work deriving no-gap second order conditions for optimal control of PDEs with polyhedric constraint sets.
Another approach to deal with SNC for problems governed by smooth quasilinear elliptic equations was followed in \cite{CasasTroltzsch2009,CasasDhamo2011-2ndOSs}. There, the non-negativity of the second-order derivatives of auxiliary real functions at minimum points (see \cite[pp.~710]{CasasTroltzsch2009}) was employed to derive SNC that have a minimal gap in comparison with the corresponding SSC \cite[Rem.~5.3.1]{CasasTroltzsch2009}. Interestingly, an inspection shows that the problems considered in \cite{CasasTroltzsch2009,CasasDhamo2011-2ndOSs} fulfill the polyhedricity condition. 

However, there are comparatively few contributions on SSC for optimal control problems governed by non-smooth PDEs, and even less on SNC for such problems.
In the literature, a common approach pursued in, e.g., \cite{KunischWachsmuth2012,Livia2019,ChristofWachsmuth2019} is to exploit a strong stationarity condition to obtain a second-order Taylor-type expansion of the mapping $u \mapsto J(S(u),u)$, where $S$ is the control-to-state operator. In these papers, SSC are derived using an additional sign assumption on the Lagrange multipliers in the vicinity of the contact set that ensures a so-called ''safety distance'' \cite[Rem.~4.13]{BetzMeyer2015}. 
In contrast, here we can use that the gradient term $\nabla \bar y$ occurring in the Taylor-type expansion \eqref{eq:2nd-Taylor-expansion} vanishes on the ``active set'' where the coefficient is non-differentiable (due to the finiteness assumption on the set of non-differentiability points of $a$). The benefit of following this approach is that we do not need any sign assumption on the multipliers. 
A related approach of deriving no-gap second order-conditions for non-smooth problems in terms of a generalized curvature functional was introduced in \cite{ChristofWachsmuth2018,ChristofWachsmuth2019}.

Finally, we mention that a generalized version of problem \eqref{eq:P} was studied in \cite{ClasonNhuRosch} to derive the Clarke, Bouligand, and strong stationarity conditions for the case where the coefficient $a$ is merely directionally differentiable and locally Lipschitz continuous. In \cite{ClasonNhuRosch}, we proved the equivalence of Clarke and strong stationarity conditions when the function $a$ is countably $PC^1$ (continuous and $C^1$ apart from countably many points). More interestingly, as we will see later, under the assumption that $a$ is finitely $PC^1$ the control-to-state operator is, indeed, first-order continuously differentiable because of the finiteness assumption and the occurrence of the divergence term in the linearized equation, cf. \cref{thm:diff} below.

\bigskip

The paper is organized as follows. \Cref{sec:assumption} is devoted to notation and the main assumptions of \eqref{eq:P}. In \cref{sec:quasilinear}, we provide some required properties of the state equation, where we present some results on the existence, uniqueness, and regularity of solutions and prove first-order differentiability of the control-to-state operator. \Cref{sec:existence-1stOS} is concerned with the existence of minimizers as well as first-order necessary optimality conditions of \eqref{eq:P}. The main results of the paper, the no-gap second-order necessary and sufficient condititions, are derived in \cref{sec:2nd-OC}. Finally, the paper ends with appendices showing an a priori estimate for \eqref{eq:P} on convex domains and verifying a central assumption on a jump functional for a one-dimensional example.

\section{Notation and main assumptions}\label{sec:assumption}

\paragraph*{Notation.}
For a given point $u\in X$ and $\rho>0$, we denote by $B_X(u,\rho)$ and $\overline B_X(u,\rho)$ the open and closed balls, respectively, of radius $\rho$ centered at $u$.
For Banach spaces $X$ and $Y$, the notation $X \hookrightarrow Y$ means that $X$ is continuously embedded in $Y$, and $X \Subset Y$ means that $X$ is compactly embedded in $Y$. For a Banach space $X$ with dual $X^*$, the symbol $\langle\cdot, \cdot \rangle_{X^*, X}$ denotes the duality pairing between $X$ and $X^*$.
For a function $f:\Omega\to \R$ defined on a domain $\Omega\subset \R^N$ and a subset $A \subset \R$, we denote by $\{ f \in A \}$ the set of all points $x \in \Omega$ for which $f(x) \in A$. Similarly, for functions $f_1,f_2$ and subsets $A_1, A_2 \subset \R$, the symbol $\{ f_1 \in A_1, f_2 \in A_2 \}$ denotes the set of all points at which the values of $f_1$ and $f_2$ belonging to $A_1$ and $A_2$, respectively. For any set $\omega \subset \Omega$, the symbol $\1_{\omega}$ stands for the indicator function of $\omega$, i.e., $\1_\omega(x) = 1$ if $x \in \omega$ and $\1_\omega(x) =0$ otherwise. Finally, $C$ denotes a generic positive constant, which may be different at different places of occurrence. We also write, e.g., $C_\xi$ for a constant depending only on the parameter $\xi$.

\medskip

We recall the following definition from, e.g., \cite[Chap.~4]{Scholtes} or \cite[Def.~2.19]{Ulbrich2011}. 
For an open subset $O$ in $\R$, we say that a continuous function $f: O \to \R$ is a \emph{$PC^k$-function}, $1 \leq k \leq \infty$, if for each point $t_0 \in O$ there exist a neighborhood $O_{t_0} \subset O$ and a finite set of $C^k$-functions $f_i: O_{t_0} \to \R$, $i=1,2,\dots,m$, such that
\begin{equation*}
    f(t) \in \left\{f_1(t), f_2(t),\dots,f_m(t)\right\} \quad \text{for all} \quad t \in O_{t_0}.
\end{equation*}
This implies in particular that $f$ is locally Lipschitz continuous; see, e.g., \cite[Cor.~4.1.1]{Scholtes}.
For a $PC^k$-function $f: O \to \R$, $1 \leq k \leq \infty$, we can thus define the exceptional set
\begin{equation*}
    E_{f} := \left\{ t \in O \,\middle|\, f \ \text{is not differentiable at } t \right\},
\end{equation*}
which by Rademacher's Theorem has Lebesgue measure zero.
We shall say that a $PC^k$-function $f$ is \emph{finitely (countably) $PC^k$} if the set $E_f$ is finite (countable); see e.g., \cite{ClasonNhuRosch}. 

\begin{example}
    The functions $\R \ni t \mapsto |t| \in \R$, $\R \ni t \mapsto \max\{0, t\} \in \R$, and $\R \ni t \mapsto \min\{0, t\} \in \R$ are finitely $PC^{\infty}$.
\end{example}

Let $f$ be a finitely $PC^1$-function on $\R$ such that the set $E_f$ is given as
\begin{equation*}
    E_f = \{t_1,t_2,\ldots, t_{\barK} \} \quad \text{with} \, -\infty < t_1 <t_2 < \cdots < t_{\barK} < \infty \, \text{and }  \barK \in \N.
\end{equation*}
For convenience, set $t_{0} := -\infty$ and $t_{\barK+1} := \infty$.
By the decomposition theorem for piecewise smooth functions \cite[Prop.~2D.7]{Dontchev-Rockafellar2014}, $f$ can then be expressed as 
\begin{equation} 
    \label{eq:PC1-rep}
    f(t) = \sum_{i =0}^{\barK} \1_{(t_{i},t_{i+1}]}(t)f_{i}(t)  \quad \text{for all } t \in \R,
\end{equation}
where $f_i$, $0 \leq i \leq \barK$, are $C^1$-functions on $\R$ such that
\begin{equation} \label{eq:PC1-continuous-cond}
    f_{i-1}(t_i) = f_{i}(t_i) \quad \text {for all } 1 \leq i \leq \barK.
\end{equation}
Here and in what follows, we use the convention $(t, +\infty] := (t, \infty)$. 
For any $i \in \{1,2,\ldots, \barK\}$, we set
\begin{equation}
    \label{eq:sigma-i}
    \sigma_i := |f'_{i-1}(t_i) - f'_{i}(t_i)|.
\end{equation}
Minding \eqref{eq:PC1-rep}, this term measures the jump of the derivative of $f$ in the singular point $t_i$ and will play an important part in the second-order optimality conditions for \eqref{eq:P}.

It is easy to see that if $f$ is a finitely $PC^1$-function defined by \eqref{eq:PC1-rep}, then it is directionally differentiable and its directional derivative is given by
\begin{equation*}
    f'(t; h) = \sum_{i=0}^{\barK} \left\{  \1_{(t_{i},t_{i+1})}(t)f_i'(t)h + \1_{\{t_{i+1}\}}(t) \left[ \1_{(0, \infty)}(h) f_{i+1}'(t_{i+1})h + \1_{(-\infty, 0)}(h) f_{i}'(t_{i+1})h\right] \right\},
\end{equation*}
where we use the convention $\1_{\{t_{\barK +1} \}} = \1_{\{ \infty \}} = 0$. 

Throughout the paper, we need the following assumptions.
\begin{assumption}
\item \label{ass:domain}
    $\Omega \subset \R^N$, $N \in \{2,3\}$, is an open, convex, and bounded domain. 
\item \label{ass:b_func}
    The function $b: \overline \Omega \to \R$ is Lipschitz continuous with Lipschitz constant $L_{b} >0$ and satisfies
    \begin{equation*}
        b(x) \geq \underline{b} > 0
    \end{equation*} 
    for all $x \in \overline \Omega$ and for some constant $\underline{b}$.
\item \label{ass:PC1-func}
    $a: \R \to \R$ is a non-negative finitely $PC^2$ and is defined by \eqref{eq:PC1-rep} with $C^2$ non-negative functions $a_i$ satisfying \eqref{eq:PC1-continuous-cond} and numbers $t_i \in \R$, $i \in \{1,2,\ldots, \barK\}$, $t_{0} := - \infty$, $t_{\barK+1} := + \infty$. 

\item \label{ass:cost_func}
    The functional $G: H^1_0(\Omega) \to \R$ is of class $C^2$.
\end{assumption}

For any $y\in C(\overline\Omega)$, we then define 
\begin{equation}\label{eq:index-set-reduced}
    I_y  := \left\{ i \in \N \,\middle|\, \exists x \in \overline\Omega \, \text{such that } y(x) \in (t_i, t_{i+1}]  \right\}.
\end{equation}
Obviously, we then have
\begin{equation*}
    a(y(x)) = \sum_{i \in I_y} \1_{(t_i,t_{i+1}]}(y(x))a_i(y(x)) \quad \text{for all } x \in \overline\Omega.
\end{equation*}
Furthermore, since the $a_i$ are $C^2$ and therefore Lipschitz continuous on bounded sets, $a$ is also Lipschitz continuous on bounded sets (where by the assumption only finitely many selections $a_i$ can be attained).

\section{Control-to-state operator} \label{sec:quasilinear}

In this section, we shall derive the required results for the state equation
\begin{equation} \label{eq:state}
    \left\{
        \begin{aligned}
            -\div [(b + a(y) )\nabla y ]& = u && \text{in } \Omega, \\
            y &=0 && \text{on } \partial\Omega.
        \end{aligned}
    \right.
\end{equation}

\subsection{Existence, uniqueness, and regularity of solutions to the state equation}

We first address the existence, uniqueness, and regularity of solutions to \eqref{eq:state}.
\begin{theorem}[cf. {\cite[Thm.~2.4 and Thm.~2.5]{CasasTroltzsch2009}}]  \label{thm:existence}
    Let $p, q > N$ be arbitrary. Let \crefrange{ass:domain}{ass:PC1-func} hold. Then, for any $u \in W^{-1,p}(\Omega)$, there exists a unique solution $y_u \in W^{1,p}_0(\Omega)$ to \eqref{eq:state}. Moreover, for any bounded set $U \in W^{-1,p}(\Omega)$, a constant $C_U$ exists such that
    \begin{equation} \label{eq:apriori}
        \|y_u\|_{W^{1,p}_0(\Omega)}  \leq C_U \quad \text{for all } u \in U.
    \end{equation}
    Moreover, if $U$ is a bounded set in $L^q(\Omega)$, then, for any $u \in U$, there holds that $y_u \in H^2(\Omega) \cap W^{1,\infty}(\Omega)$ and 
    \begin{equation} \label{eq:apriori_2}
        \|y_u\|_{H^2(\Omega)}  + \|y_u\|_{W^{1,\infty}(\Omega)}  \leq C_{U}.
    \end{equation}
\end{theorem}
\begin{proof}
    Let $p >N$, $U$ be a bounded set in $W^{-1,p}(\Omega)$, and $u \in U$ be arbitrary. By \cite[Thm.~2.2]{CasasTroltzsch2009}, \eqref{eq:state} has a unique solution $y_u \in H^1_0(\Omega) \cap C(\overline\Omega)$, and there exists a constant $C_{1,U}$ such that
    \begin{equation}
        \label{eq:apriori-auxi1}
        \|y_u\|_{H^{1}_0(\Omega)} + \|y_u\|_{C(\overline\Omega)} \leq C_{1,U} \quad \text{for all } u \in U.
    \end{equation}
    We now use the Kirchhoff transformation (see \cite[Chap.~V]{Visintin})
    \begin{equation}
        \label{eq:Kirchoff-trans}
        K(x,t) := b(x)t + \int_0^t a(s) \ds.
    \end{equation}
    By setting $\theta_u(x) := K(x, y_u(x))$ for $x \in \overline\Omega$, \eqref{eq:state} can be rewritten as follows
    \begin{equation} \label{eq:state-Kirchhoff}
        \left\{
            \begin{aligned}
                -\Delta \theta_u &= u - \div (\nabla b y_u) && \text{in } \Omega, \\
                \theta_u &=0 && \text{on } \partial\Omega.
            \end{aligned}
        \right.
    \end{equation}
    Applying the maximal elliptic regularity for Poisson's equation on the convex domain  (see, e.g., \cite[Cor.~1]{Fromm1993}) to \eqref{eq:state-Kirchhoff} yields that
    \begin{equation*}
        \|\theta_u\|_{W^{1,p}_0(\Omega)}  \leq C_{\Omega,p,N} \norm{u - \div (\nabla b y_u)}_{W^{-1,p}(\Omega)}.
    \end{equation*}
    This, together with \eqref{eq:apriori-auxi1} and the global Lipschitz continuity of $b$, gives
    \begin{equation}
        \label{eq:apriori-auxi2}
        \|\theta_u\|_{W^{1,p}_0(\Omega)}  \leq C_{2,U}  \quad \text{for all } u \in U
    \end{equation}
    and for some constant $C_{2,U}$. Moreover, for any fixed $x \in \Omega$, $K(x,\cdot)$ is monotonically increasing due to \cref{ass:b_func,ass:PC1-func}. It then has an inverse denoted by $T(x, \cdot)$. By a simple computation, we have for all $1 \leq i \leq N$ that
    \begin{equation*}
        \left|\frac{\partial y_u}{\partial x_i} \right| = \left| \frac{\partial T}{\partial x_i} + \frac{\partial T}{\partial s} \frac{\partial \theta_u}{\partial x_i}   \right|  \leq \frac{1}{\underline{b}}\left[ \left|\frac{\partial b}{\partial x_i} \right| |y_u| + \left|\frac{\partial \theta_u}{\partial x_i} \right| \right].
    \end{equation*}
    From this, \eqref{eq:apriori-auxi2}, and \eqref{eq:apriori-auxi1}, we derive \eqref{eq:apriori}. 

    It remains to prove \eqref{eq:apriori_2}. To this end, let $U$ be a bounded set in $L^q(\Omega)$ with $q >N$ and take $u \in U$ arbitrary but fixed. Since $q > N$, we have the continuous embedding $L^q(\Omega) \hookrightarrow W^{-1,2q}(\Omega)$. This gives $y_u \in W^{1,2q}(\Omega)$. We thus have the $H^2$- and $W^{1,\infty}$-regularity of $y_u$ as well as \eqref{eq:apriori_2} according to \eqref{eq:apriori-auxi1} and \cref{lem:regularity-convex}.
\end{proof}

From now on, for each $u \in W^{-1,p}(\Omega)$, $p >N$, we denote by $y_u$ the unique solution to \eqref{eq:state}. The control-to-state operator $ W^{-1,p}(\Omega) \ni u \mapsto y_u \in W^{1,p}_0(\Omega)$ is denoted by $S$, which is uniformly bounded by \cref{thm:existence}.

\subsection{Differentiability of the control-to-state operator}

We now prove the first-order differentiability of the control-to-state operator even for the non-differentiable coefficient $a$. To this end, we will employ the differentiability of the implicit mapping \cite[Thm.~2.1]{Wachsmuth2014}, which is a generalized version of the classical implicit function theorem \cite[Chap.~4]{ZeidlerI} and applies to a class of quasilinear PDEs.

We first derive the locally Lipschitz continuity of the control-to-state mapping $S$.
\begin{lemma} \label{lem:Lip}
    Let $p >N$ and $u \in W^{-1,p}(\Omega)$ be arbitrary. Let \crefrange{ass:domain}{ass:PC1-func} hold.  Then the operator $S$ is locally Lipschitz continuous at $u$ as a function from $W^{-1,p}(\Omega)$ to $W^{1,p}_0(\Omega)$. Moreover, for any bounded set $U$ in $W^{-1,p}(\Omega)$, there exists a constant $L_U$ such that
    \begin{equation}
        \label{eq:Lip-esti}
        \norm{S(u_1) - S(u_2)}_{W^{1,p}_0(\Omega)} \leq L_U \norm{u_1 - u_2}_{W^{-1,p}(\Omega)} \quad \text{for all } u_1, u_2 \in U.
    \end{equation}
\end{lemma}
\begin{proof}
    It is enough to prove \eqref{eq:Lip-esti}.
    Let $u_1, u_2 \in U$ be arbitrary and set $y_i:= S(u_i)$ and $\theta_i(x) := K(x,y_i(x))$, $i=1,2$, with $K$ defined in \eqref{eq:Kirchoff-trans}. Similar to \eqref{eq:state-Kirchhoff}, we have
    \begin{equation} \label{eq:state-Kirchhoff-Lip}
        \left\{
            \begin{aligned}
                -\Delta (\theta_1 - \theta_2) &= u_1 -u_2 - \div [\nabla b (y_1 - y_2)] && \text{in } \Omega, \\
                \theta_1 - \theta_2 &=0 && \text{on } \partial\Omega.
            \end{aligned}
        \right.
    \end{equation}
    Applying \cite[Cor.~1]{Fromm1993} to \eqref{eq:state-Kirchhoff-Lip} and using the fact that $\norm{\nabla b}_{L^\infty(\Omega)} \leq L_{b}$ yields
    \begin{align} \label{eq:theta-esti-w1p}
        \|\theta_1 -\theta_2 \|_{W^{1,p}_0(\Omega)} 
        & \leq C_{\Omega, p, N} \left[ \norm{u_1 - u_2}_{W^{-1,p}(\Omega)}  + L_{b}\norm{ y_1 -y_2}_{L^{p}(\Omega)}  \right].
    \end{align}
    By the definition of $\theta_i$, $i=1,2$, it follows for all $x \in \overline \Omega$ that
    \begin{equation*}
        \theta_1(x) - \theta_2(x) = b(x) (y_1(x) - y_2(x)) + \int_{y_2(x)}^{y_1(x)} a(s) \ds.
    \end{equation*}
    From this and a straightforward computation, we derive for all $1 \leq m \leq N$ that
    \begin{equation} \label{eq:Lip-deri}
        \frac{\partial}{\partial x_m}(y_1 - y_2) = \frac{1}{b + a(y_1)} \left[\frac{\partial}{\partial x_m}(\theta_1 - \theta_2) - \frac{\partial b}{\partial x_m}(y_1 - y_2) - (a(y_1) - a(y_2))\frac{\partial y_2}{\partial x_m}   \right].
    \end{equation}
    For almost every $x \in \Omega$, since $K(x, \cdot)$ is monotonically increasing, so is its inverse $T(x,\cdot)$. This implies for almost every $x \in \Omega$ that $\theta_1(x) \geq \theta_2(x)$ if and only if $y_1(x) \geq y_2(x)$. Consequently, we obtain 
    \begin{equation*}
        |\theta_1(x) - \theta_2(x) | = b(x) |y_1(x) - y_2(x)| + \left| \int_{y_2(x)}^{y_1(x)} a(s) \ds  \right| \geq \underline{b}|y_1(x) - y_2(x)|.
    \end{equation*} 
    From this, the continuous embedding $W^{1,p}_0(\Omega) \hookrightarrow C(\overline\Omega)$, and \eqref{eq:theta-esti-w1p}, there holds
    \begin{equation} \label{eq:Lip-infty-esti}
        \|y_1 -y_2 \|_{L^\infty(\Omega)} \leq C_{\Omega, p, N, L_{b}, \underline{b}}\left[ \norm{u_1 - u_2}_{W^{-1,p}(\Omega)}  + \norm{ y_1 -y_2}_{L^{p}(\Omega)}  \right].
    \end{equation}
    Furthermore, as a result of \cref{thm:existence} and the continuous embedding $W^{1,p}_0(\Omega) \hookrightarrow C(\overline\Omega)$, there exists a constant $C_U >0$ such that
    \begin{equation}
        \label{eq:bounds}
        \norm{y_i}_{C(\overline\Omega)}, \norm{y_i}_{W^{1,p}_0(\Omega)} \leq C_{1,U} \quad \text{for } i=1,2.
    \end{equation}
    The combination of \eqref{eq:Lip-deri} with \eqref{eq:theta-esti-w1p}, \eqref{eq:Lip-infty-esti}, \eqref{eq:bounds}, and \cref{ass:b_func,ass:PC1-func} implies that
    \begin{equation*}
        \|y_1 -y_2 \|_{W^{1,p}_0(\Omega)} \leq C_{2,U} \left[ \norm{u_1 - u_2}_{W^{-1,p}(\Omega)}  + \norm{ y_1 -y_2}_{L^{p}(\Omega)}  \right].
    \end{equation*}
    Combining this with Young's inequality and the continuous embedding $W^{1,p}_0(\Omega) \hookrightarrow L^\infty(\Omega)$, there holds for all $\epsilon >0$ that
    \begin{align*}
        \|y_1 -y_2 \|_{W^{1,p}_0(\Omega)} & \leq C_{2,U}  \left[ \norm{u_1 - u_2}_{W^{-1,p}(\Omega)}  + \norm{ y_1 -y_2}_{L^{\infty}(\Omega)}^{(p-1)/p} \norm{ y_1 -y_2}_{L^{1}(\Omega)}^{1/p}  \right] \\
        & \leq C_{2,U}  \left[ \norm{u_1 - u_2}_{W^{-1,p}(\Omega)}  + \epsilon^{p/(p-1)}\norm{ y_1 -y_2}_{L^{\infty}(\Omega)} + \frac{1}{\epsilon^{p}} \norm{ y_1 -y_2}_{L^{1}(\Omega)}  \right] \\
        & \leq C_{2,U}  \left[ \norm{u_1 - u_2}_{W^{-1,p}(\Omega)}  + \epsilon^{p/(p-1)}\norm{ y_1 -y_2}_{W^{1,p}_0(\Omega)} + \frac{1}{\epsilon^{p}} \norm{ y_1 -y_2}_{L^{1}(\Omega)}  \right].
    \end{align*}
    By choosing $\epsilon = \epsilon(p,C_{2,U}) >0$ small enough, we arrive at
    \begin{equation} \label{eq:Lip-semi}
        \|y_1 -y_2 \|_{W^{1,p}_0(\Omega)} \leq C_{U} \left[ \norm{u_1 - u_2}_{W^{-1,p}(\Omega)}  + \norm{ y_1 -y_2}_{L^{1}(\Omega)}  \right].
    \end{equation}
    We now show that there is a constant $L_U$ satisfying
    \begin{equation}
        \label{eq:w-esti}
        \norm{y_1 -y_2}_{L^{2p/(p-2)}(\Omega)} \leq L_U  \norm{u_1 - u_2}_{W^{-1,p}(\Omega)} \quad \text{for all } u_1, u_2 \in U,
    \end{equation}
    which, together with \eqref{eq:Lip-semi}, gives the desired conclusion. Assume to the contrary that \eqref{eq:w-esti} does not hold. Then we can find $u_1^{(n)}, u_2^{(n)} \in U$ such that 
    \begin{equation*}
        \frac{1}{\eta_n} \norm{y_1^{(n)} - y_2^{(n)} }_{L^{2p/(p-2)}(\Omega)} \to +\infty
    \end{equation*}
    with $\eta_n := \norm{u_1^{(n)} - u_2^{(n)}}_{W^{-1,p}(\Omega)}$ and $y_i^{(n)} := S(u_i^{(n)})$, $ i =1,2$. Obviously, $\eta_n \to 0$ as $n \to \infty$.
    We now define a scalar function $a_{n}$ on $\overline\Omega$ and a vector-valued function $\mathbb{b}_{n}$ on $\Omega$ by 
    \begin{equation*}
        {a}_{n}(x) := b(x) + a(y_1^{(n)}(x)),\qquad
        \mathbb{b}_{n}(x) := \nabla y_2^{(n)}(x)\1_{\{y_1^{(n)} \neq y_2^{(n)}\}}(x) \frac{a(y_1^{(n)}(x)) - a(y_2^{(n)}(x))}{y_1^{(n)}(x)- y_2^{(n)}(x)}.
    \end{equation*} 
    As a result of \eqref{eq:apriori}, a constant $C_U$ exists such that
    \begin{equation} \label{eq:bounded-bn}
        \norm{y_i^{(n)}}_{W^{1,p}_0(\Omega)}, \norm{\mathbb b_n}_{L^p(\Omega)} \leq C_U \quad \text{for all } n \in\N, i = 1,2.
    \end{equation}
    Setting $w_n := y_1^{(n)} - y_2^{(n)}$ yields
    \begin{equation*} 
        \left\{
            \begin{aligned}
                -\div [a_{n} \nabla w_n + \mathbb{b}_{n}w_n ]& = u_1^{(n)} - u_2^{(n)} && \text{in } \Omega, \\
                w_n &=0 && \text{on } \partial\Omega.
            \end{aligned}
        \right.
    \end{equation*}
    Setting
    \begin{equation*}
        \rho_n := \frac{\eta_n}{\norm{y_1^{(n)} - y_2^{(n)} }_{L^{2p/(p-2)}(\Omega)} } , \qquad
        \xi_n := \frac{\rho_n}{\eta_n} w_n,
        \qquad
        h_n := \frac{\rho_n}{\eta_n} \left[u_1^{(n)} - u_2^{(n)} \right],
    \end{equation*}
    we see that $\rho_n \to 0$ and that $\xi_n$ solves
    \begin{equation} \label{eq:xi-auxi}
        \left\{
            \begin{aligned}
                -\div [a_{n} \nabla \xi_n + \mathbb{b}_{n} \xi_n ]& = h_n && \text{in } \Omega, \\
                \xi_n &=0 && \text{on } \partial\Omega.
            \end{aligned}
        \right.
    \end{equation}
    Testing the above equation by $\xi_n$ and employing the Hölder inequality yield
    \begin{equation*}
        \underline{b} \norm{\nabla \xi_n}_{L^2(\Omega)}^2  \leq \norm{h_n}_{H^{-1}(\Omega)}\norm{\xi_n}_{H^1_0(\Omega)} + \norm{\mathbb{b}_{n}}_{L^p(\Omega)} \norm{\xi_n}_{L^{2p/(p-2)}}\norm{\nabla \xi_n}_{L^2(\Omega)},
    \end{equation*}  
    which, together with \eqref{eq:bounded-bn} and the fact that $\norm{\xi_n}_{H^1_0(\Omega)} \leq C\norm{\nabla \xi_n}_{L^2(\Omega)}$, gives
    \begin{equation} \label{eq:H1-esti}
        \begin{aligned}[t]
            \underline{b} \norm{\nabla \xi_n}_{L^2(\Omega)}  &\leq \norm{h_n}_{H^{-1}(\Omega)} + C\norm{\mathbb{b}_{n}}_{L^p(\Omega)} \norm{\xi_n}_{L^{2p/(p-2)}} \\
            & \leq C \norm{h_n}_{W^{-1,p}(\Omega)} + C_U \\
            & = C \rho_n + C_U \\
            & \leq C + C_U
        \end{aligned}
    \end{equation}  
    for $n$ large enough. From this and the compact embedding $H^{1}_0(\Omega) \Subset L^{2p/(p-2)}(\Omega)$, we can assume that 
    $\xi_n \rightharpoonup \xi$ in $H^1_0(\Omega)$ and $\xi_n \to \xi$ in $L^{2p/(p-2)}(\Omega)$ for some $\xi \in H^1_0(\Omega)$. Moreover, there exist subsequences of $\{y_1^{(n)}\}$ and $\{\mathbb{b}_n \}$, denoted in the same way, such that $y_1^{(n)} \to y_*$ in $C(\overline\Omega)$ and $\mathbb{b}_n \rightharpoonup \mathbb{b}$ in $L^p(\Omega)^N$ for some $y_* \in C(\overline\Omega)$ and $\mathbb{b} \in L^p(\Omega)^N$. Therefore, we have $a_n \to a_*$ in $C(\overline\Omega)$ with $a_*(x) := b(x) +a(y_*(x))$. Passing to the limit in \eqref{eq:xi-auxi}, we deduce from the fact that $h_n \to 0$ in $H^{-1}(\Omega)$ that $\xi$ fulfills
    \begin{equation*} 
        \left\{
            \begin{aligned}
                -\div [a_* \nabla \xi + \mathbb{b} \xi ]& = 0 && \text{in } \Omega, \\
                \xi &=0 && \text{on } \partial\Omega.
            \end{aligned}
        \right.
    \end{equation*} 
    The uniqueness of solutions, see, e.g., \cite[Thm.~2.6]{CasasDhamo2011}, implies that $\xi = 0$, which contradicts the fact that $\norm{\xi}_{L^{2p/(p-2)}(\Omega)} = \lim_{n \to \infty}\norm{\xi_n}_{L^{2p/(p-2)}(\Omega)} = 1$.
\end{proof}

For any $y, \hat y \in C(\overline\Omega)$ and any $\tau_1, \tau_2 \in \R$ we define the set
\begin{equation}
    \label{eq:Omega-sets}
    \Omega_{\hat y, i, j }^{[\tau_1, \tau_2]} := \left\{ \hat y \in [t_i + \tau_1, t_j + \tau_2 ]  \right\},
\end{equation}
where $t_i$, $i \in \{0,1,\ldots, \barK+1\}$, are given in \cref{ass:PC1-func}. Similar sets such as $\Omega_{\hat y, i, j}^{[\tau_1, \tau_2)}$ are defined in the same way.
We also define the function $T_{y, \hat y}: \Omega \to \R$ via
\begin{equation}
    \label{eq:T-func}
    T_{y, \hat y} := \1_{\{\hat y \notin E_{a}\} }\left[ a(y) - a(\hat y) - a'(\hat y)(y -\hat y) \right].
\end{equation}
From now on, let us fix a number $\delta \in R$ such that 
\begin{equation*}
    0 < \delta \leq \frac{t_{i+1} - t_i}{2} \quad \text{for all } 1 \leq i \leq \barK -1.
\end{equation*} 
We need the following lemmas.
\begin{lemma}\label{lem:T-decomposition}
    Let \cref{ass:PC1-func} be satisfied. Then, for any $y, \hat y \in C(\overline\Omega)$ with $\norm{y - \hat y}_{C(\overline\Omega)} < \delta$, there holds
    \begin{equation}
        \label{eq:T-decomposition}
        T_{y, \hat y} = \sum_{i \in I_{\hat y}} \left( T_{y, \hat y}^{i,1} + T_{y, \hat y}^{i,2} + T_{y, \hat y}^{i,3}\right)
    \end{equation}
    with $I_{\hat y}$ defined via \eqref{eq:index-set-reduced} and
    \begin{align*}
        & T_{y, \hat y}^{i,1} := \1_{ \Omega_{y, \hat y }^{i,1} }\left[ a_i(y) - a_i(\hat y) - a_i'(\hat y)(y -\hat y) \right],\\
        & T_{y, \hat y}^{i,2} := \1_{ \Omega_{y, \hat y }^{i,2} }\left[ a_{i-1}(y) - a_i(\hat y) - a_i'(\hat y)(y -\hat y) \right],\\
        & T_{y, \hat y}^{i,3} := \1_{ \Omega_{y, \hat y }^{i,3} }\left[ a_{i+1}(y) - a_i(\hat y) - a_i'(\hat y)(y -\hat y) \right], 
    \end{align*}
    where 
    \begin{align*}
        &\Omega_{y, \hat y }^{i,1} := \Omega_{\hat y, i, i +1}^{[\delta, -\delta]} \cup \left( \Omega_{\hat y, i, i }^{(0, \delta)} \cap \Omega_{y, i, i }^{(0, 2\delta)}  \right) \cup \left( \Omega_{\hat y, i+1, i+1}^{(-\delta, 0)} \cap \Omega_{y, i+1, i+1 }^{(-2\delta, 0)}  \right), \\
        & \Omega_{y, \hat y }^{i,2} := \Omega_{\hat y, i, i }^{(0, \delta)} \cap \Omega_{y, i, i }^{(-\delta, 0]}, \quad \Omega_{y, \hat y }^{i,3} := \Omega_{\hat y, i+1, i+1 }^{(-\delta, 0)} \cap \Omega_{y, i+1, i+1 }^{[0,\delta)}. 
    \end{align*}
    Moreover, if $y_n \to \hat y$ in $W^{1,p}_0(\Omega)$ with $p>N$, then 
    \begin{equation}
        \label{eq:T-limit}
        \frac{1}{\norm{y_n - \hat y}_{W^{1,p}_0(\Omega)}} \norm{ T_{y_n, \hat y} \nabla \hat y}_{L^p(\Omega)} \to 0.
    \end{equation}
\end{lemma}
\begin{proof}
    We have 
    \begin{equation}\label{eq:T-decomposition1}
        \begin{aligned}[t]
            T_{y, \hat y} & = \sum_{i = 0}^{\barK} \1_{\Omega_{\hat y, i,i+1}^{(0,0)}} \left[ a(y) - a_i(\hat y) - a_i'(\hat y)( y -\hat y) \right]\\
            & = \sum_{i \in I_{\hat y}} \left[\1_{\Omega_{\hat y, i,i}^{(0,\delta)}} + \1_{\Omega_{\hat y, i,i+1}^{[\delta,-\delta]}} +\1_{\Omega_{\hat y, i+1,i+1}^{(-\delta,0)}} \right]\left[ a(y) - a_i(\hat y) - a_i'(\hat y)( y -\hat y) \right],
        \end{aligned}
    \end{equation} 
    using the fact that $\1_{\Omega_{\hat y, i,i+1}^{(0,0)}} \equiv 0$ for all $i \notin I_{\hat y}$.
    Since $\norm{y - \hat y}_{C(\overline\Omega)} < \delta$, we have
    \begin{align*}
        \Omega_{\hat y, i,i+1}^{[\delta,-\delta]} &= \Omega_{\hat y, i,i+1}^{[\delta,-\delta]} \cap \Omega_{y, i,i+1}^{(0,0)},\\
        \Omega_{\hat y, i,i}^{(0,\delta)} &= \Omega_{\hat y, i,i}^{(0,\delta)} \cap \left( \Omega_{y, i, i }^{(-\delta, 0]} \cup  \Omega_{y, i, i }^{(0, 2\delta)}\right),\\
        \Omega_{\hat y, i+1,i+1}^{(-\delta,0)} &= \Omega_{\hat y, i+1,i+1}^{(-\delta,0)} \cap \left(\Omega_{y, i+1, i+1 }^{(-2\delta, 0)} \cup \Omega_{y, i+1, i+1 }^{[0,\delta)}\right).
    \end{align*}
    Together with \eqref{eq:T-decomposition1}, these yield the claimed expression.

    It remains to prove the limit \eqref{eq:T-limit}. Let $y_n \to \hat y$ in $W^{1,p}_0(\Omega)$ and set $\epsilon_n := \norm{y_n - \hat y}_{C(\overline\Omega)}$. Then $\epsilon_n \to 0$ since $W^{1,p}_0(\Omega) \hookrightarrow C(\overline\Omega)$. We therefore can assume that $\epsilon_n < \delta$  and $|y_n(x)| \leq M$ for some $M>0$ and all $x \in \overline\Omega$ for $n\in \N$ sufficiently large. Using \eqref{eq:T-decomposition}, we obtain
    \begin{equation}
        \label{eq:T-esti}
        \norm{T_{y_n, \hat y} \nabla \hat y}_{L^p(\Omega)} \leq \sum_{i \in I_{\hat y} } \left( \norm{T_{y_n, \hat y}^{i,1}\nabla \hat y}_{L^p(\Omega)} + \norm{T_{y_n, \hat y}^{i,2}\nabla \hat y}_{L^p(\Omega)} + \norm{T_{y_n, \hat y}^{i,3}\nabla \hat y}_{L^p(\Omega)} \right).
    \end{equation}
    From the definition of $T_{y_n,\hat y}^{i,1}$, the dominated convergence theorem yields
    \begin{equation}
        \label{eq:T-limit2}
        \frac{1}{\epsilon_n} \norm{T_{y_n, \hat y}^{i,1}\nabla \hat y}_{L^p(\Omega)} \to 0 \quad \text{for all } i \in I_{\hat y}.
    \end{equation}
    For $\frac{1}{\epsilon_n} \norm{T_{y_n, \hat y}^{i,2}\nabla \hat y}_{L^p(\Omega)}$, we have
    \begin{equation*}
        \begin{aligned}
            |a_{i-1}(y_n) - a_i(\hat y) - a_i'(\hat y)(y_n - \hat y)| 
            & = |a_{i-1}(y_n) - a_{i-1}(t_i) + a_i(t_i) - a_i(\hat y) - a_i'(\hat y)(y_n - \hat y)|\\
            & \leq L_i\left( |y_n - t_i| + |t_i - \hat y| + | y_n - \hat y| \right) \\
            & = 2L_i | y_n - \hat y|  \\
            & \leq 2 L_i \epsilon_n
        \end{aligned} 
    \end{equation*} 
    for almost every $x\in \Omega_{y_n, \hat y }^{i,2}$ and for some constant $L_i$ depending on $M$. Moreover, $\1_{\Omega_{y_n, \hat y }^{i,2}} \to 0$ almost everywhere in $\Omega$. Combining this with the dominated convergence theorem yields
    \begin{equation}
        \label{eq:T-limit3}
        \frac{1}{\epsilon_n} \norm{T_{y_n, \hat y}^{i,2}\nabla \hat y}_{L^p(\Omega)} \to 0 \quad \text{for all } i \in I_{\hat y}.
    \end{equation}
    Similarly, 
    \begin{equation}
        \label{eq:T-limit4}
        \frac{1}{\epsilon_n} \norm{T_{y_n, \hat y}^{i,3}\nabla \hat y}_{L^p(\Omega)} \to 0 \quad \text{for all } i \in I_{\hat y}.
    \end{equation}
    From \eqref{eq:T-esti}--\eqref{eq:T-limit4} and the fact that $\epsilon_n \leq C\norm{y_n - \hat y}_{W^{1,p}_0(\Omega)}$ and that $I_{\hat y}$ is finite, we obtain the limit \eqref{eq:T-limit}.   
\end{proof}

\begin{lemma}
    \label{lem:smooth-der}
    Let $y_n \to y$ in $W^{1,p}_0(\Omega)$ with $p>N$. Under \cref{ass:PC1-func}, there holds
    \begin{equation*}
        \1_{\{y_n \notin E_{a} \}} a'(y_n) \nabla y_n  - \1_{\{y \notin E_{a} \}} a'(y) \nabla y  \to 0 \quad \text{in } L^p(\Omega).
    \end{equation*}
\end{lemma}
\begin{proof}
    Since $y_n \to y$ in $W^{1,p}_0(\Omega)$, we have $y_n \to y$ in $C(\overline\Omega)$. Also, we can assume that
    $\norm{y_n - y}_{C(\overline\Omega)} < \delta$ for all $n \in\N$ large enough.
    Setting $A_n := \1_{\{y_n \notin E_{a} \}} a'(y_n) \nabla y_n  - \1_{\{y \notin E_{a} \}} a'(y) \nabla y $, we have 
    \begin{equation*}
        A_n = \sum_{i = 0}^{\barK} A_n^{(i)} \quad\text{with}\quad
        A_n^{(i)} := \1_{\Omega_{y_n, i, i+1 }^{(0,0)}}a_i'(y_n) \nabla y_n - \1_{\Omega_{y, i, i+1 }^{(0,0)}}a_i'(y) \nabla y.
    \end{equation*}
    It is sufficient to prove that $A_n^{(i)} \to 0$ in $L^p(\Omega)$ for all $0 \leq i \leq \barK$. To this end, we write $A_n^{(i)}$ as
    \begin{equation*}
        A_n^{(i)} = \left[\1_{\Omega_{y_n, i, i+1 }^{(0,0)}} - \1_{\Omega_{y, i, i+1 }^{(0,0)}} \right]  a_i'(y) \nabla y
        + \1_{\Omega_{y_n, i, i+1 }^{(0,0)}}\left[a_i'(y_n)- a_i'(y) \right] \nabla y + \1_{\Omega_{y_n, i, i+1 }^{(0,0)}}a_i'(y_n) \left[ \nabla y_n - \nabla y \right].
    \end{equation*}
    The last two terms in the right-hand side tend to $0$ in $L^p(\Omega)$. For the first term in the right-hand side, we have
    \begin{multline*}
        \left[\1_{\Omega_{y_n, i, i+1 }^{(0,0)}} - \1_{\Omega_{y, i, i+1 }^{(0,0)}} \right]  a_i'(y) \nabla y = \1_{\Omega_{y_n, i, i+1}^{(0,0)} \cap \left[\{y= t_i\} \cup \{y = t_{i+1}\} \right] } a_i'(y) \nabla y  \\
        \begin{aligned}
            & + \1_{ \Omega_{y_n, i, i}^{(0,\delta)} \cap \Omega_{y, i, i }^{(-\delta, 0)} }   a_i'(y) \nabla y 
            + \1_{ \Omega_{y_n, i+1, i+1}^{(-\delta,0)} \cap \Omega_{y, i+1, i+1}^{(0, \delta)} }   a_i'(y) \nabla y \\
            &- \1_{ \Omega_{y_n, i, i}^{(-\delta,0]} \cap \Omega_{y, i, i}^{(0, \delta)} }   a_i'(y) \nabla y - \1_{ \Omega_{y_n, i+1, i+1}^{[0, \delta)} \cap \Omega_{y, i+1, i+1}^{(-\delta, 0)} }   a_i'(y) \nabla y.
        \end{aligned}
    \end{multline*}
    In the above expression, the first term in the right-hand side vanishes almost everywhere since $\nabla y(x) = 0$ for almost every $x\in\{y = t_i\} \cup \{y =t_{i+1} \}$ (see, e.g., \cite[Rem.~2.6]{Chipot2009} and \cite[Prop.~5.8.2]{Attouch2006}), and the other terms tend to zero in $L^p(\Omega)$ according to the dominated convergence theorem. Thus, 
    \begin{equation*}
        \left[\1_{\Omega_{y_n, i, i+1 }^{(0,0)}} - \1_{\Omega_{y, i, i+1 }^{(0,0)}} \right]  a_i'(y) \nabla y \to 0 \qquad\text{in } L^p(\Omega)
    \end{equation*}
    and hence $A_n^{(i)} \to 0$ in $L^p(\Omega)$.
\end{proof}

As seen in the proof of \cref{lem:smooth-der}, the fact that for $y \in W^{1,p}(\Omega)$, $p\geq 1$, the gradient $\nabla y$ vanishes almost everywhere on $\{y \in E_{a} \}$, plays an important role.
Furthermore, this fact also guarantees that
\begin{equation*}
    [a(y+z) - a(y) - a'(y;z)]\nabla y = T_{y+z,y} \nabla y   
\end{equation*}
for $y, z \in W^{1,p}_0(\Omega)$, which will be crucial for proving continuous differentiability of the control-to-state mapping $S$. We will also use this to show a key limit in \cref{lem:limits}\,(iii) below. 
We point out that do \emph{not} need to assume that the set $\{y\in E_{a}\}$ has small or even zero Lebesgue measure.
\begin{theorem}
    \label{thm:diff}
    Under \crefrange{ass:domain}{ass:PC1-func}, the control-to-state operator $S: W^{-1,p}(\Omega) \to W^{1,p}_0(\Omega)$, $p >N$, is Fr\'{e}chet differentiable. Moreover, for any $u,v \in W^{-1,p}(\Omega)$, $z_v:= S'(u)v$ satisfies
    \begin{equation} \label{eq:deri}
        \left\{
            \begin{aligned}
                -\div [(b + a(y_u)) \nabla z_v + \1_{\{ y_u \notin E_{a} \}}a'(y_u) z_v\nabla y_u ]& = v && \text{in } \Omega, \\
                z_v &=0 && \text{on } \partial\Omega,
            \end{aligned}
        \right.
    \end{equation}
    with $y_u := S(u)$.
    Furthermore, the mapping $W^{-1,p}(\Omega) \ni u\mapsto S'(u) \in W^{1,p}_0(\Omega)$ is continuous.
\end{theorem}
\begin{proof}
    \emph{Step 1: existence of $S'$.}
    We shall apply the differentiability of the implicit mapping \cite[Thm.~2.1]{Wachsmuth2014}. Consider the mapping $F: W^{1,p}_0(\Omega) \times W^{-1,p}(\Omega) \to W^{-1,p}(\Omega)$ defined by
    \begin{equation*}
        F(y,u) := - \div \left[(b + a(y)) \nabla y \right] - u.
    \end{equation*}
    We first show that $F$ has a partial derivative in $y$ given by
    \begin{equation*}
        \frac{\partial F}{\partial y}(y,u)z = - \div \left[(b + a(y)) \nabla z + \1_{\{ y \notin E_{a} \}}a'(y) z\nabla y \right]
    \end{equation*}
    for all $y, z \in W^{1,p}_0(\Omega)$ and $u \in W^{-1,p}(\Omega)$. To this end, let $z_n$ be an arbitrary sequence converging to zero in $W^{1,p}_0(\Omega)$ and let $\varphi$ be arbitrary in $W^{1,p'}_0(\Omega)$, $p' := p/(p-1)$, with $\norm{\varphi}_{W^{1,p'}_0(\Omega)} \leq 1$. Setting $y_n := y + z_n$, we then have
    \begin{multline*}
        \left\langle F(y_n,u) - F(y,u) -  \frac{\partial F}{\partial y}(y,u)z_n, \varphi \right \rangle_{W^{-1,p}(\Omega), W^{1,p'}_0(\Omega) }\\
        \begin{aligned}
            & = \int_\Omega \left[(b+ a(y_n)) \nabla y_n - (b + a(y))\nabla y\right. \\
            \MoveEqLeft[-3] \left.- (b+ a(y)) \nabla z_n - a'(y;z_n) \nabla y \right]\cdot \nabla \varphi \dx\\
            & = \int_\Omega \left[(a(y_n)-a(y) ) \nabla z_n + \left( a(y_n) - a(y) -  a'(y;z_n) \right) \nabla y \right]\cdot \nabla \varphi \dx \\
            & \leq \norm{(a(y_n)-a(y) ) \nabla z_n + \left( a(y_n) - a(y) -  a'(y;z_n) \right) \nabla y}_{L^p(\Omega)} \\
            & \leq \norm{a(y_n) - a(y)}_{L^\infty(\Omega)} \norm{\nabla z_n}_{L^p(\Omega)} + \norm{\left( a(y_n) - a(y) -  a'(y;z_n) \right) \nabla y}_{L^p(\Omega)} \\
            & = \norm{a(y_n) - a(y)}_{L^\infty(\Omega)} \norm{\nabla z_n}_{L^p(\Omega)} + \norm{ T_{y_n, y} \nabla y}_{L^p(\Omega)}.
        \end{aligned}
    \end{multline*}
    Setting $\epsilon_n := \norm{z_n}_{W^{1,p}_0(\Omega)}$, we obtain
    \begin{equation*}
        \frac{1}{\norm{z_n}_{W^{1,p}_0(\Omega)}} \norm{F(y_n,u) - F(y,u) -  \frac{\partial F}{\partial y}(y,u)z_n}_{W^{-1,p}(\Omega)}  
        \leq \norm{a(y_n) - a(y)}_{L^\infty(\Omega)} + \frac{1}{\epsilon_n} \norm{ T_{y_n, y} \nabla y}_{L^p(\Omega)}.
    \end{equation*}
    Obviously, $\norm{a(y_n) - a(y)}_{L^\infty(\Omega)} \to 0$. Combining this with \eqref{eq:T-limit} yields the partial differentiability of $F$ in $y$. 

    We now prove that $\frac{\partial F}{\partial y}(y_u,u)$, $y_u := S(u)$, is an isomorphism as a mapping from $W^{1,p}_0(\Omega)$ to $W^{-1,p}(\Omega)$. From this, the Lipschitz continuity of $S$ (see \cref{lem:Lip}), and \cite[Thm.~2.1]{Wachsmuth2014}, we then deduce the existence of Fr\'{e}chet derivative at $u$ of $S$ as well as \eqref{eq:deri}. It is enough to prove for any $v \in W^{-1,p}(\Omega)$ that there exists a unique $z_v \in W^{1,p}_0(\Omega)$ such that
    \begin{equation}
        \label{eq:F-isomorphism}
        \frac{\partial F}{\partial y}(y_u,u) z_v = v
    \end{equation} 
    and 
    \begin{equation}
        \label{eq:F-iso-esti}
        \norm{z_v}_{W^{1,p}_0(\Omega)} \leq C_u \norm{v}_{W^{-1,p}(\Omega)}.
    \end{equation}
    Setting $a_u(x) := b(x) + a(y_u(x))$ and $\mathbb{c}_u(x) := \1_{\{ y_u \notin E_{a}\} }(x) a'(y_u(x)) \nabla y_u(x)$ for almost every $x \in \Omega$, we have that $a_u$ is continuous and there hold
    \begin{equation}
        \label{eq:coef-cond}
        \underline{b} \leq a_u(x) \leq C_u, \quad \norm{\mathbb{c}_u }_{L^p(\Omega)} \leq C_u \quad \text{for a.e. } x \in \Omega.
    \end{equation}
    The equation \eqref{eq:F-isomorphism} can thus be written as
    \begin{equation} 
        \label{eq:F-isomorphism2}
        \left\{
            \begin{aligned}
                -\div [a_{u} \nabla z_v + \mathbb{c}_u z_v ]& = v && \text{in } \Omega, \\
                z_v &=0 && \text{on } \partial\Omega.
            \end{aligned}
        \right.
    \end{equation} 
    Due to \cite[Thm.~2.6]{CasasDhamo2011}, \eqref{eq:F-isomorphism2} (and thus \eqref{eq:F-isomorphism}) has a unique solution $z_v \in H^1_0(\Omega) \hookrightarrow L^{2p/(p-2)}(\Omega)$. 
    Similar to \eqref{eq:w-esti} and \eqref{eq:H1-esti}, there hold that
    \begin{align} \label{eq:zv-esti1}
        \norm{z_v}_{L^{2p/(p-2)}(\Omega)} &\leq C_{1,u} \norm{v}_{W^{-1,p}(\Omega)}\\
        \intertext{and}
        \underline{b}\norm{\nabla z_v}_{L^{2}(\Omega)} &\leq C_{1,u} \left[ \norm{v}_{W^{-1,p}(\Omega)}  +  \norm{z_v}_{L^{2p/(p-2)}(\Omega)} \right]\notag
    \end{align}
    for some constant $C_{1,u}$ independent of $v$. It then follows that
    \begin{equation} \label{eq:zv-esti-H1}
        \norm{\nabla z_v}_{L^{2}(\Omega)} \leq C_{2,u}\norm{v}_{W^{-1,p}(\Omega)}.
    \end{equation}
    Moreover, by using the chain rule and the product formula \cite[Chap.~7]{Gilbarg_Trudinger} as well as  \cref{ass:PC1-func}, we can rewrite \eqref{eq:F-isomorphism2} as
    \begin{equation} 
        \label{eq:F-isomorphism3}
        \left\{
            \begin{aligned}
                -\Delta \hat z_v& = v - \div [z_v\nabla b] && \text{in } \Omega, \\
                z_v &=0 && \text{on } \partial\Omega,
            \end{aligned}
        \right.
    \end{equation}
    with $\hat z_v := (b + a(y_u))z_v$. We now show that
    \begin{equation}\label{eq:zv-Linfty-esti}
        \norm{\hat z_v}_{L^\infty(\Omega)}  \leq C_{\Omega,p}\norm{v}_{W^{-1,p}(\Omega)}.
    \end{equation}
    To this end, we first consider the case $p \geq 6$. We then have $v \in W^{-1,p}(\Omega) \hookrightarrow W^{-1,6}(\Omega)$. Besides, we have $z_v \in H^1_0(\Omega) \hookrightarrow L^6(\Omega)$ and thus $z_v\nabla b \in L^{6}(\Omega)^N$. It follows that the right-hand side of \eqref{eq:F-isomorphism3} belongs to $W^{-1,6}(\Omega)$. Applying the Stampacchia Theorem \cite[Thm.~12.4]{Chipot2009} and using the continuous embedding $W^{-1,p}(\Omega) \hookrightarrow W^{-1,6}(\Omega)$, we can conclude that
    \begin{equation*}
        \begin{aligned}[t]
            \norm{\hat z_v}_{L^\infty(\Omega)} & \leq C_{\Omega}\left[\norm{v}_{W^{-1,6}(\Omega)} + \norm{z_v\nabla b}_{L^6(\Omega)} \right] \\
            & \leq C_{\Omega,p}\left[\norm{v}_{W^{-1,p}(\Omega)} + \norm{\nabla b}_{L^\infty(\Omega)} \norm{z_v}_{H^1_0(\Omega)} \right] \\
            & \leq C_{\Omega,p}\norm{v}_{W^{-1,p}(\Omega)}, 
        \end{aligned}
    \end{equation*}
    where we have used \eqref{eq:zv-esti-H1} to obtain the last estimate. For the case $N < p < 6$, we see that the right-hand side of \eqref{eq:F-isomorphism3} belongs to $W^{-1,p}(\Omega)$. Using the Stampacchia Theorem again and exploiting the embedding $H^1_0(\Omega) \hookrightarrow L^{6}(\Omega) \hookrightarrow L^{p}(\Omega)$ and \eqref{eq:zv-esti-H1} yield
    \begin{equation*}
        \begin{aligned}[t]
            \norm{\hat z_v}_{L^\infty(\Omega)} & \leq C_{\Omega,p}\left[\norm{v}_{W^{-1,p}(\Omega)} + \norm{z_v\nabla b}_{L^p(\Omega)} \right] \\
            & \leq C_{\Omega,p}\left[\norm{v}_{W^{-1,p}(\Omega)} + \norm{z_v\nabla b}_{L^6(\Omega)} \right] \\
            & \leq C_{\Omega,p}\norm{v}_{W^{-1,p}(\Omega)}.
        \end{aligned}
    \end{equation*}        
    We now consider two cases:
    \begin{enumerate}[label={Case }\arabic*:, align=left]
        \item $N =2$. Since the embedding $H^1_0(\Omega) \hookrightarrow L^p(\Omega)$ is continuous, we have    
            $v - \div [z_v\nabla b] \in W^{-1,p}(\Omega)$. Applying \cite[Cor.~1]{Fromm1993} to \eqref{eq:F-isomorphism3} and exploiting \eqref{eq:zv-esti-H1} yields that $\hat z_v \in W^{1,p}_0(\Omega)$ and that
            \begin{equation*}
                \norm{\hat z_v}_{W^{1,p}_0(\Omega)} \leq C_{\Omega, p, N, u}  \norm{v}_{W^{-1,p}(\Omega)}. 
            \end{equation*}
            From this, the definition of $\hat z_v$, and \eqref{eq:zv-Linfty-esti}, we derive $z_v \in W^{1,p}_0(\Omega)$ and \eqref{eq:F-iso-esti}.

        \item $N=3$. In this case, we have $H^1_0(\Omega) \hookrightarrow L^6(\Omega)$. If $p \leq 6$, we then have the desired conclusion using a similar argument as in the first case. If $p >6$, then $v - \div [z_v\nabla b] \in W^{-1,6}(\Omega)$. Similar to the first case, $z_v \in W^{1,6}_0(\Omega)$ and $\norm{z_v}_{W^{1,6}_0(\Omega)} \leq C_u \norm{v}_{W^{-1,p}(\Omega)}$. Finally, by using the continuous embedding $W^{1,6}_0(\Omega) \hookrightarrow L^p(\Omega)$ and a bootstrapping argument, we arrive at the desired conclusion.
    \end{enumerate}

    \medskip 

    \emph{Step 2. continuity of $S'$.} Taking any $v, u_n, u_0 \in W^{-1,p}(\Omega)$ such that $\norm{v}_{W^{-1,p}(\Omega)} \leq 1$ as well as $s_n := \norm{u_n - u_0}_{W^{-1,p}(\Omega)} \to 0$ as $n \to \infty$ and setting $z_{vn}:= S'(u_n)v, z_{v} := S'(u_0)v$, we see that $z_{vn}$ and $z_v$ satisfy
    \begin{equation*} 
        \label{eq:S-smooth}
        \left\{
            \begin{aligned}
                -\div [(b + a(y_0)) \nabla (z_{vn} - z_v)  + \1_{\{y_0 \notin E_{a} \}} a'(y_0) \nabla y_0 (z_{vn} - z_v) ]& = \div g_{vn} && \text{in } \Omega, \\
                z_{vn} - z_v &=0 && \text{on } \partial\Omega,
            \end{aligned}
        \right.
    \end{equation*}
    with $y_0 := S(u_0)$, $y_n := S(u_n)$, and 
    \begin{equation*}
        g_{vn} := \left[ \nabla z_{vn} (a(y_n) - a(y_0)) + \left( \1_{\{y_n \notin E_{a} \}} a'(y_n) \nabla y_n  - \1_{\{y_0 \notin E_{a} \}} a'(y_0) \nabla y_0  \right)z_{vn} \right].
    \end{equation*}
    Similar to \eqref{eq:F-iso-esti}, a constant $C_1 := C_{1u_0}$ exists such that
    \begin{equation}
        \label{eq:subs-esti}
        \norm{z_{vn} - z_v}_{W^{1,p}_0(\Omega)} \leq C_1 \norm{g_{vn}}_{L^p(\Omega)}.
    \end{equation}
    Furthermore, the local Lipschitz continuity of $S$ (\cref{lem:Lip}) implies that there is a  constant $C_{2} = C_{2}(u_0)$ such that
    \begin{equation*}
        \norm{z_{vn}}_{W^{1,p}_0(\Omega)}, \norm{z_{v}}_{W^{1,p}_0(\Omega)} \leq C_{2}\norm{v}_{W^{-1,p}(\Omega)} \leq C_2 \quad \text{for all } n \in\N.
    \end{equation*}
    This implies that
    \begin{equation}
        \label{eq:g-esti}
        \begin{aligned}[t]
            \norm{g_{vn}}_{L^p(\Omega)} & \leq C_2 \norm{a(y_n) - a(y_0)}_{L^\infty(\Omega)} + \norm{z_{vn}}_{L^\infty(\Omega)} \norm{A_n}_{L^p(\Omega)}\\
            & \leq C_2 \left( \norm{a(y_n) - a(y_0)}_{L^\infty(\Omega)} + C \norm{A_n}_{L^p(\Omega)} \right),
        \end{aligned}
    \end{equation}
    where we have used the continuous embedding $W^{1,p}_0(\Omega) \hookrightarrow L^\infty(\Omega)$ to obtain the last inequality with a constant $C$ independent of $v$ and $n$. Here
    \begin{equation*}
        A_n:= \1_{\{y_n \notin E_{a} \}} a'(y_n) \nabla y_n  - \1_{\{y_0 \notin E_{a} \}} a'(y_0) \nabla y_0 
    \end{equation*}
    does not depend on $v$. Combining \eqref{eq:subs-esti} with \eqref{eq:g-esti} yields
    \begin{equation*}
        \norm{S'(u_n) - S'(u_0)}_{\Linop(W^{-1,p}(\Omega), W^{1,p}_0(\Omega))} \leq C_3 \left[ \norm{a(y_n) - a(y_0)}_{L^\infty(\Omega)} + \norm{A_n}_{L^p(\Omega)} \right].
    \end{equation*}
    Obviously, the first term in the right-hand side converges to zero since $y_n \to y_0$ in $C(\overline\Omega)$. In addition, as a result of \cref{lem:smooth-der}, the second term tends to zero. 
\end{proof}

We end this section with a direct consequence of the differentiability of $S$. 
\begin{corollary}
    \label{cor:diff}
    Let $u$ and $h$ be arbitrary in $L^2(\Omega)$ and let $\{s_n\} \subset (0, \infty)$ and $\{h_n\} \subset L^2(\Omega)$ be such that $s_n \to 0^+$ and $h_n \rightharpoonup h$ in $L^2(\Omega)$.
    Then, for any $p \in (N, 6)$, there holds
    \begin{equation*}
        \frac{S(u + s_nh_n) - S(u)}{s_n} \to S'(u)h \quad \text{in } W^{1,p}_0(\Omega) \hookrightarrow C(\overline\Omega).
    \end{equation*}
\end{corollary}
\begin{proof}
    We write
    \begin{equation*}
        \frac{S(u + s_nh_n) - S(u)}{s_n} - S'(u)h = \frac{S(u + s_nh_n) - S(u) -s_n S'(u)h_n}{s_n} + S'(u)(h_n-h).
    \end{equation*}
    From this, \cref{thm:diff}, and the compact embedding  $L^2(\Omega) \Subset W^{-1,p}(\Omega)$ and the continuous embedding $W^{1,p}_0(\Omega) \hookrightarrow C(\overline\Omega)$, we derive the desired conclusion.
\end{proof}

\begin{remark}
    All results in this and the following sections can be extended to the case where an additional semilinear term $f(x,y(x))$ is present if $f$ is a Carathéodory function and for almost every $x\in \Omega$, the mapping $z\mapsto f(x,z)$ is monotone and of class $C^2$ as in, e.g., \cite{CasasTroltzsch2009}. However, in order to avoid additional technicalities, we have restricted the discussion to the case \eqref{eq:state}.
    Similarly, in assumption \ref{ass:domain}, the convexity of $\Omega$ was only used to establish the $H^2$- and $W^{1,\infty}$-regularity of solutions of elliptic equations. The results therefore remain valid for other domains guaranteeing that regularity, e.g., if $\Omega$ is of class $C^{1,1}$.
\end{remark}

\section{Existence and first-order optimality conditions} \label{sec:existence-1stOS}

The optimal control problem \eqref{eq:P} can be rewritten in the form
\begin{equation}\label{eq:P2}
    \tag{P}
    \left\{
        \begin{aligned}
            \min_{u\in L^\infty(\Omega)} & j(u) := J(S(u),u) = G(S(u)) + \frac\nu2\norm{u}_{L^2(\Omega)}^2 \\
            \text{s.t.} &  \quad u \in \mathcal{U}_{ad}, 
        \end{aligned}
    \right.
\end{equation}
where the admissible set is defined by 
\begin{equation}
    \label{eq:ad-set}
    \mathcal{U}_{ad} := \left\{u \in L^\infty(\Omega) \,\middle|\, \alpha(x) \leq u(x) \leq \beta(x) \quad \text{for a.e. } x \in \Omega  \right\}.
\end{equation}
The existence of a minimizer $\bar u\in \mathcal{U}_{ad}$ of \eqref{eq:P2} follows as in \cite[Thm.~3.1]{CasasTroltzsch2009}. 

To derive first-order necessary optimality conditions, we first address the existence, uniqueness, and regularity of solutions to the adjoint state equation
\begin{equation} \label{eq:adjoint-state}
    \left\{
        \begin{aligned}
            -\div [(b + a(y_u)) \nabla \varphi]  + \1_{\{ y_u \notin E_{a} \}}a'(y_u) \nabla y_u \cdot \nabla \varphi & = v && \text{in } \Omega, \\
            \varphi &=0 && \text{on } \partial\Omega,
        \end{aligned}
    \right.
\end{equation}
for $u \in W^{-1,p}(\Omega)$, $p >N$, $v\in H^{-1}(\Omega)$, and $y_u:= S(u)$.
\begin{lemma}
    \label{lem:adjoint-equation} 
    Let \crefrange{ass:domain}{ass:PC1-func} hold and $p>N$. Then, for any $u \in W^{-1,p}(\Omega)$ and $v \in H^{-1}(\Omega)$, there exists a unique $\varphi \in H^1_0(\Omega)$ solving \eqref{eq:adjoint-state}. If, in addition, $U$ is a bounded subset in $L^p(\Omega)$, then for any $u \in U$ and $v \in L^q(\Omega)$ with $q >N$, there exists a constant $C_U$ such that $\varphi \in H^2(\Omega) \cap W^{1,\infty}(\Omega)$ satisfies
    \begin{equation}
        \label{eq:apriori-adjoint2}
        \norm{\varphi}_{H^2(\Omega)} + \norm{\varphi}_{W^{1,\infty}(\Omega)} \leq C_{U} \norm{v}_{L^{q}(\Omega)}.
    \end{equation}
\end{lemma} 
\begin{proof} 
    For any $u \in W^{-1,p}(\Omega)$, we define a function $a_u$ and a vector-valued function $\mathbb{c}_u$ given by
    \begin{equation*}
        a_u := b + a(y_u), \qquad \mathbb{c}_u := \1_{\{ y_u \notin E_{a} \}}a'(y_u) \nabla y_u.
    \end{equation*}
    By \cref{thm:existence}, we have $y_u \in W^{1,p}_0(\Omega)$ and thus $a_u \in W^{1,p}(\Omega)$ and $\mathbb c_u \in L^p(\Omega)^N$.
    Consider the operator
    \begin{equation*}
        T_u: H^1_0(\Omega)  \to  H^{-1}(\Omega) \qquad
        \varphi  \mapsto  T_u\varphi := -\div [a_u \nabla \varphi]  + \mathbb{c}_u \cdot \nabla \varphi. 
    \end{equation*}
    Then $T_u$ is an isomorphism (cf.~\cite[Thm.~2.6]{CasasDhamo2011}). Therefore, for any $v \in H^{-1}(\Omega)$, there exists a unique solution $\varphi \in H^1_0(\Omega)$ to \eqref{eq:adjoint-state} such that $T_u \varphi = v$.

    It remains to show the $W^{1,\infty}$-regularity of $\varphi$ and the estimate \eqref{eq:apriori-adjoint2}. Let $U$ be a bounded subset in $L^p(\Omega)$ and $u \in U$, $v \in L^q(\Omega)$ be arbitrary.
    Then $y_u \in W^{1,\infty}(\Omega)$ by \cref{thm:existence}. It follows that $\mathbb{c}_u \in L^\infty(\Omega)^N$ and $a_u$ is uniformly Lipschitz continuous. In addition, $\varphi \in H^{1}_0(\Omega)$ satisfies
    \begin{equation*}
        -\div [a_u \nabla \varphi] = v - \mathbb{c}_u \cdot \nabla \varphi \in L^2(\Omega),
    \end{equation*}
    which together with the $H^{2}$-regularity of solutions (see, e.g., \cite[Thm.~3.2.1.2]{Grisvard1985}) gives $\varphi \in H^2(\Omega)$. Therefore, $\varphi \in H^{1}_0(\Omega) \cap H^2(\Omega)$ is the strong solution to 
    \begin{equation*}
        - \Delta \varphi = \frac{1}{a_u} \left[  v - \mathbb{c}_c \cdot \nabla \varphi + \nabla a_u \cdot \nabla \varphi \right] \quad \text{in } \Omega.
    \end{equation*}
    By virtue of the chain rule (see, e.g., \cite[Thm.~7.8]{Gilbarg_Trudinger}), one has $\nabla a_u = \nabla b + \mathbb{c}_u$. We then have that $\varphi \in H^{1}_0(\Omega) \cap H^2(\Omega)$ satisfies
    \begin{equation} \label{eq:varphi}
        - \Delta \varphi = f := \frac{1}{a_u} \left[  v + \nabla b \cdot \nabla \varphi \right] \quad \text{in } \Omega
    \end{equation}
    and there holds
    \begin{equation*}
        \norm{\Delta \varphi}_{L^2(\Omega)} \leq \frac{1}{\underline{b}} \left[\norm{v}_{L^2(\Omega)} + L_{b} \norm{\varphi}_{H^{1}_0(\Omega)} \right].
    \end{equation*}
    On the other hand, using a similar argument as in \eqref{eq:zv-esti-H1} and employing the continuous embedding $L^q(\Omega) \hookrightarrow W^{-1,q}(\Omega)$, we can show that
    \begin{equation*}
        \norm{\varphi}_{H^1_0(\Omega)} \leq C_{1,U}\norm{v}_{W^{-1,q}(\Omega)} \leq C_{2,U}\norm{v}_{L^{q}(\Omega)}.
    \end{equation*}
    We thus have
    \begin{equation} \label{eq:H2-esti}
        \norm{\varphi}_{H^2(\Omega)} \leq C_{3,U} \norm{v}_{L^{q}(\Omega)}.
    \end{equation}
    Setting $\tilde{q} := \min\{q,6\} >N$ and using the continuous embedding $H^2(\Omega) \hookrightarrow W^{1,6}(\Omega) \hookrightarrow W^{1,\tilde{q}}(\Omega)$ thus gives 
    \begin{equation*}
        \norm{f}_{L^{\tilde{q}}(\Omega)} \leq C_{4,U} \norm{v}_{L^{q}(\Omega)}.
    \end{equation*} 
    From this, \cref{ass:domain}, and the global boundedness of the gradient of solutions to Poisson's equation (see, e.g. \cite[Thm.~3.1, Rem.~3]{CianchiMazya2015}, \cite{Mazya2009,YangChen2018}), we can conclude that $\varphi \in W^{1, \infty}(\Omega)$ and that
    \begin{equation*}
        \norm{\nabla\varphi}_{L^{\infty}(\Omega)} \leq C_{5,U} \norm{f}_{L^{\tilde{q}}(\Omega)} \leq C_{5,U}C_{4,U} \norm{v}_{L^{q}(\Omega)},
    \end{equation*}
    which, along with \eqref{eq:H2-esti}, gives \eqref{eq:apriori-adjoint2}.
\end{proof}

From the chain rule, \cref{lem:adjoint-equation}, and an elementary calculus, we derive the following result.
\begin{theorem}
    \label{thm:objective-deri}
    Let \crefrange{ass:domain}{ass:cost_func} hold. Then the reduced cost functional $j: L^2(\Omega) \to \R$ is of class $C^1$. Moreover, for any $u, h\in L^2(\Omega)$, there holds
    \begin{equation*}
        j'(u)h = \int_{\Omega} \left(\varphi_u + \nu u \right)h \dx,
    \end{equation*}
    where $\varphi_u \in H^{1}_0(\Omega)$ solves \eqref{eq:adjoint-state} corresponding to the right-hand side term $v=G'(y_u)$ and $y_u$ solves \eqref{eq:state}. 
\end{theorem}

We now arrive at first-order necessary optimality condition for the problem \eqref{eq:P2}. Since the reduced functional is Fréchet differentiable by \cref{thm:objective-deri}, the proof of the following result, based on the variational inequality $j'(\bar u)(u -\bar u) \geq 0$ for all $u \in \mathcal{U}_{ad}$, is standard and therefore omitted.
\begin{theorem}
    \label{thm:1st-OC}
    Let \crefrange{ass:domain}{ass:cost_func} hold. If $\bar u$ is a local minimizer of \eqref{eq:P2}, then there exists an adjoint state $\bar\varphi \in H^1_0(\Omega)$ such that
    \begin{subequations}
        \label{eq:1st-OS}
        \begin{align}
            &\left\{
                \begin{aligned}
                    -\div \left [\left(b +a(\bar y) \right) \nabla \bar\varphi \right] + \1_{\{ \bar y \notin E_{a}\} } a'(\bar y) \nabla \bar y \cdot \nabla \bar \varphi &= G'(\bar y) && \text{in } \Omega, \\
                    \bar \varphi &=0 && \text{on } \partial\Omega,
                \end{aligned}
            \right.  \label{eq:adjoint_OS} \\
            &\int_\Omega \left(\bar \varphi + \nu \bar u \right)\left( u - \bar u \right) \dx \geq 0 \quad \text{for all } u \in \mathcal{U}_{ad}  \label{eq:normal_OS}
        \end{align}
    \end{subequations}
    with $\bar y:= S(\bar u)$.
\end{theorem}

\begin{remark} \label{rem:Pontryagin}
    If the discrepancy term $G$ is an integral functional of the form
    \begin{equation*}
        G(y)  = \int_\Omega g(x,y(x)) \dx
    \end{equation*}
    where $g: \Omega \times \R \to \R$ is a Carath\'{e}odory function of class $C^1$ with respect to the second variable such that \cref{ass:cost_func} holds,
    then  we obtain from \cref{thm:1st-OC} and \cite[Lem.~2.26]{Troltzsch2010} the Pontryagin maximum principle
    \begin{equation*}
        g(x,\bar y(x))+ \tfrac{\nu}{2} \bar u(x)^2 +\bar\varphi(x)\bar u(x) = \min \left\{g(x,\bar y(x))+\tfrac{\nu}{2} s^2 +\bar\varphi(x) s \,\middle|\, {s \in [\alpha(x), \beta(x)]} \right\} \quad \text{for a.e. } x \in \Omega
    \end{equation*}
    due to the convexity of the mapping $u\mapsto J(y,u)$; cf. \cite[Thm.~4.1]{CasasTroltzsch2009}. 
\end{remark}

\section{Second-order optimality conditions} \label{sec:2nd-OC}

Our goal is now to derive second-order necessary and sufficient conditions in terms of a non-smooth \emph{curvature functional} characterizing the (generalized) curvature of the reduced functional $j$ in critical directions. A similar approach was followed in \cite{ChristofWachsmuth2018,ChristofWachsmuth2019}.
We will introduce the necessary technical notation in \cref{sec:2nd-OC:curvature}, prove preliminary estimates in \cref{sec:2nd-OC:estimates}, and finally derive the desired second-order conditions in \cref{sec:2nd-OC:soc}.

\subsection{Non-smooth curvature functional}\label{sec:2nd-OC:curvature}

Intuitively, differentiating \eqref{eq:P} (formally) and applying the sum and product rules, we see that the total curvature of $j$ can be separated into three contributions: a \emph{smooth} part involving only $a$ and its derivatives in smooth points; a \emph{first-order non-smooth} part involving only first (directional) derivatives of $a$; and a \emph{second-order non-smooth} part relating to second generalized derivatives of $a$.

Correspondingly, we define the following three partial curvature functionals at a point $(u,y,\varphi)\in L^2(\Omega)\times H^1(\Omega)\times W^{1,\infty}(\Omega)$.
First, the smooth part of the curvature in directions $(h_1,h_2)\in L^2(\Omega)^2$ is given by
\begin{multline}\label{eq:curvature-smooth}
    Q_s(u,y,\varphi; h_1,h_2) := \frac{1}{2}G''(y)(S'(u)h_1 S'(u)h_2) + \frac{\nu}{2} \int_\Omega h_1h_2 \dx\\
    - \frac{1}{2} \int_\Omega\1_{ \{ y \notin E_{a} \}} a''(y)(S'(u)h_1)(S'(u)h_2) \nabla y \cdot \nabla \varphi \dx,
\end{multline}
which is a bilinear form in $(h_1,h_2)$.

The first-order non-smooth part of the curvature is given by
\begin{equation*}\label{eq:curvature-nonsmooth-1}
    Q_1(u,y,\varphi; h_1,h_2) :=  
    - \frac{1}{2} \int_\Omega \left[a'(y; S'(u)h_1) \nabla (S'(u)h_2) +a'(y; S'(u)h_2)\nabla (S'(u)h_1) \right]\cdot  \nabla\varphi  \dx.
\end{equation*}
Although $Q_1(u, y, \varphi; \cdot, \cdot)$ is not bilinear, it is positively homogeneous in each variable due to the positive homogeneity of the function $a'(y(x); \cdot)$ for all $x \in \overline \Omega$, i.e.,
\begin{equation*}
    Q_1(u, y, \varphi;\tau_1h_1, \tau_2 h_2) = \tau_1\tau_2 Q_1(u, y, \varphi; h_1, h_2) \quad \text{for all } h_1, h_2 \in L^2(\Omega), \tau_1, \tau_2 \geq 0.
\end{equation*}
If $a$ is a $C^2$ function, $Q_s+Q_1$ corresponds to the second derivative $j''(u)(h_1,h_2)$ of the reduced functional;
in this case our second-order conditions reduce to the results obtained in \cite{CasasTroltzsch2009}.

The critical part for our analysis is of course the second-order non-smooth part, which requires some additional notation.
For any $0 \leq i \leq \barK $, $s \in \R$, and $h \in L^2(\Omega)$, we define 
\begin{equation*}
    \begin{aligned}
        \zeta_i(u,y;s,h) &:=   [a'_{i-1}(t_i) - a'_i(t_i)] \1_{\Omega_{S(u+sh), y }^{i,2}}(t_i - S(u+sh)) \\
        \MoveEqLeft[-1]+ [a'_{i+1}(t_{i+1}) - a'_i(t_{i+1})] \1_{\Omega_{S(u+sh), y }^{i,3}}(t_{i+1}-S(u+sh)),
    \end{aligned}
\end{equation*}
with $t_i$ as defined in \cref{ass:PC1-func} and the sets $\Omega_{y, \hat y }^{i,2}, \Omega_{y, \hat y}^{i,3}$ as defined in \cref{lem:T-decomposition}. Here we use the convention $a_{-1} \equiv a_{\barK+1} \equiv 0$, $a'_{0}(t_{0}) =a'_{0}(-\infty) = 0$, and $a'_{\barK}(t_{\barK+1}) =a'_{\barK}(\infty) = 0$.
We then define for positive null sequences $\{s_n\} \in c_0^+:= \left\{ \{s_n\}\subset (0,\infty) \mid s_n \to 0 \right\}$, $h \in L^2(\Omega)$, and $\{h_n\} \subset L^2(\Omega)$, 
\begin{equation}
    \label{eq:A-n}
    A_n (u,y;\{s_n\}, \{h_n\}) := \sum_{i =0}^{\barK } \zeta_i(u,y;s_n,h_n)
\end{equation}
as well as 
\begin{equation}
    \label{eq:key-term-sn}
    \tilde{Q}(u,y,\varphi;\{s_n\}, h) := \liminf_{n \to \infty} \frac{1}{s_n^2} \int_\Omega  A_n (u,y;\{s_n\}, \{h\}) \nabla y \cdot \nabla \varphi \dx,
\end{equation}
where $\{h\}$ denotes the constant sequence $h_n \equiv h$.
(These terms will all be of independent use in the following.)
The second-order non-smooth part of the curvature in direction $h\in L^2(\Omega)$ is then given by
\begin{equation}
    \label{eq:key-term}
    Q_2(u,y,\varphi; h) := \inf \left\{ \tilde{Q}(u,y,\varphi; \{s_n\}, h) \mid \{s_n\} \in c_0^+ \right\}.
\end{equation}
Crucially, $Q_2$ is positively homogeneous of degree $2$ in $h$.
\begin{lemma} \label{lem:homogeneity}
    Let $u, h \in L^2(\Omega)$, $y \in H^1(\Omega)$, and $\varphi \in W^{1,\infty}(\Omega)$.  If $Q_2(u,y,\varphi;h) < \infty$, then 
    \begin{equation*}
        Q_2(u,y,\varphi;th) = t^2 Q_2(u,y,\varphi; h)\qquad\text{for all }t>0.
    \end{equation*}
\end{lemma}
\begin{proof}
    We first observe that
    \begin{equation*}
        \zeta_i(u,y;s_n,th) = \zeta_i(u,y;ts_n,h) \quad \text{and} \quad A_n (u,y;\{s_n\}, \{th\}) = A_n (u,y;\{ts_n\}, \{h\})
    \end{equation*}
    hold for all $0\leq i \leq \barK$, $n \geq 1$ and for all $u, h \in L^2(\Omega)$, $y \in H^1(\Omega)$, $\{s_n\} \in c_0^+$, and $t>0$. By the definition of $\tilde{Q}$, we have    
    \begin{equation*}
        \tilde{Q}(u,y,\varphi;\{s_n\}, th) = t^2 \tilde{Q}(u,y,\varphi;\{ts_n\}, h)
    \end{equation*}
    for all $u, h \in L^2(\Omega)$, $y \in H^1(\Omega)$, $\varphi \in W^{1,\infty}(\Omega)$, $\{s_n\} \in c_0^+$, and $t>0$. 
    It follows that
    \begin{equation*}
        \begin{aligned}[b]
            Q_2(u,y,\varphi; th) &= \inf \left\{ \tilde{Q}(u,y,\varphi; \{s_n\}, th) \mid \{s_n\} \in c_0^+ \right\}\\
            & = \inf \left\{ t^2\tilde{Q}(u,y,\varphi; \{ts_n\}, h) \mid \{s_n\} \in c_0^+ \right\}\\
            & = t^2 \inf \left\{ \tilde{Q}(u,y,\varphi; \{r_n\}, h) \mid \{r_n\} \in c_0^+ \right\} \quad (\text{by setting } r_n := ts_n)\\
            & = t^2 Q_2(u,y,\varphi; h).
        \end{aligned}
        \qedhere
    \end{equation*}
\end{proof}

Note also that $\zeta_i(u,y; s,h)=0$ and therefore $Q_2(u,y,\varphi; h)=0$ if the functions $a_i$ have the property that $a_{i-1}'(t_i) = a_i'(t_i)$ for all $1\leq i \leq \barK$, i.e., if the derivative $a'$ is finitely $PC^1$. 
We also remark that $Q_2$ is related to the term $\sigma(h)$ used to bound the second derivative of the Lagrangian in \cite[Chap.~3]{BonnansShapiro2000} and which there characterized the gap between necessary and sufficient second-order conditions.

\bigskip

Finally, to account for the control constraints, we recall the following basic notation standard in the study of second-order conditions; see, e.g.,\cite{BonnansShapiro2000,BayenBonnansSilva2014}.
Let $K$ be a closed subset in $L^2(\Omega)$ and let $z \in K$ be arbitrary.
The \emph{radial}, \emph{contingent tangent}, and \emph{normal} cones to $K$ at $z$ are defined, respectively, as
\begin{align*}
    & \mathcal{R}(K; z) := \left\{ v \in L^2(\Omega)\,\middle|\, \exists \bar s > 0 \text{ s.t. } z + s v \in K\ \text{for all } s \in [0,\bar s] \right\},\\
    & \mathcal{T}(K; z) := \left\{ v \in L^2(\Omega)\,\middle|\, \exists s_n \to  0^+, v_n \to v \text{ in } L^2(\Omega) \text{ s.t. } z + s_n v_n \in K \text{ for all } n \in\N  \right\},\\
    & \mathcal{N}(K; z) := \left\{ w \in L^2(\Omega) \,\,\middle|\, \int_\Omega w(x)\left( v(x) - z(x) \right)\dx \leq 0 \text{ for all } v \in \mathcal{T}(K; z) \right\}.
\end{align*} 
It is well-known that when $K$ is convex, we have
\begin{equation*}
    \mathcal{T}(K; z) = \text{cl}_2\left[ \mathcal{R}(K; z) \right],
\end{equation*}
where cl$_2(U)$ stands for the closure of a set $U$ in $L^2(\Omega)$.
For any $w \in L^2(\Omega)$, we denote the annihilator of $w$ by
\begin{equation*}
    w^\bot := \left\{ h \in L^2(\Omega) \,\middle|\, \int_\Omega w(x)h(x)\dx = 0 \right\}.
\end{equation*}
Furthermore, we say that the set $K$ is \emph{polyhedric at $z \in K$} if for any $w \in \mathcal{N}(K;z)$, there holds
\begin{equation*}
    \text{cl}_2\left[ \mathcal{R}(K; z) \cap (w^\bot) \right] = \mathcal{T}(K; z) \cap (w^\bot).
\end{equation*}
We say that $K$ is \emph{polyhedric} if it is polyhedric at each point $z \in K$.

In the following, we will consider the admissible set $\mathcal{U}_{ad}$, defined in \eqref{eq:ad-set}, as a subset in $L^2(\Omega)$ rather than a subset in $L^\infty(\Omega)$. In this case, $\mathcal{U}_{ad}$ is polyhedric, see \cite[Lem.~4.13]{BayenBonnansSilva2014}.

Furthermore, for a given triple $(\bar u, \bar y, \bar \varphi)$ with $\bar u \in \mathcal{U}_{ad}$ satisfying the system \eqref{eq:1st-OS}, set 
\begin{equation}
    \label{eq:objective-der}
    \bar d:= \bar\varphi + \nu \bar u.  
\end{equation}
Obviously, $- \bar d \in \mathcal{N}(\mathcal{U}_{ad},\bar u)$ by \eqref{eq:normal_OS}. 
The \emph{critical cone} of the problem \eqref{eq:P2} at $\bar u$ is then defined via
\begin{equation}
    \label{eq:critical-cone}
    \mathcal{C}(\mathcal{U}_{ad};\bar u) := \left\{ h \in L^2(\Omega) \,\middle|\, h \in \mathcal{T}(\mathcal{U}_{ad}; \bar u) \cap (\bar d^\bot) \right\}.
\end{equation}

By \cite[Lem.~4.11]{BayenBonnansSilva2014}, the tangent cone $\mathcal{T}(\mathcal{U}_{ad}; \bar u)$ and the critical cone $\mathcal{C}(\mathcal{U}_{ad};\bar u)$ can, respectively, be characterized pointwise as 
\begin{align}
    \mathcal{T}(\mathcal{U}_{ad}; \bar u) =& \left\{ v \in L^2(\Omega) \,\middle|\, v(x)  \begin{cases}
            \geq 0 & \text{if }  \bar u(x) = \alpha(x)\\
            \leq 0 & \text{if }  \bar u(x) = \beta(x)
        \end{cases} 
    \text{a.e. } x \in \Omega \right\}\notag
    \shortintertext{and }
    \label{eq:critical-cone2}
    \mathcal{C}(\mathcal{U}_{ad};\bar u) :=& \left\{ h \in L^2(\Omega) \,\middle|\, h(x) 
        \begin{cases}
            \geq 0 & \text{if }  \bar u(x) = \alpha(x)\\
            \leq 0 & \text{if }  \bar u(x) = \beta(x) \\
            = 0 & \text{if }  \bar d(x) \neq 0
        \end{cases} 
        \text{a.e. } x \in \Omega 
    \right\}.
\end{align}

\subsection{Preliminary estimates}\label{sec:2nd-OC:estimates}

Throughout this section, let $(\bar u, \bar y, \bar \varphi)$ be a point that satisfies the system \eqref{eq:1st-OS} and $\bar d$ be given by \eqref{eq:objective-der}. 
We start this section with a second-order Taylor-type expansion.
\begin{lemma}
    \label{lem:2nd-expression}
    For any $u \in L^2(\Omega)$ and $y_u:= S(u)$, there holds
    \begin{multline} \label{eq:2nd-Taylor-expansion}
        j(u) - j(\bar u) = \int_0^1 (1-s)G''(\bar y + s(y_u -\bar y))(y_u - \bar y)^2 \ds + \frac{\nu}{2} \norm{u - \bar u}_{L^2(\Omega)}^2 \\
        \begin{aligned}[t]
            & + \int_\Omega \bar d(u - \bar u)\dx - \int_\Omega \left[a(y_u) - a(\bar y) \right] \nabla \bar \varphi \cdot \left( \nabla y_u - \nabla \bar y \right)\dx\\
            & - \int_\Omega \left[ a(y_u) - a(\bar y) - \1_{\{\bar y \notin E_{a} \} }a'(\bar y)(y_u - \bar y) \right] \nabla \bar \varphi \cdot \nabla \bar y \dx. 
        \end{aligned}
    \end{multline}
\end{lemma}
\begin{proof}
    Since $j$ is Fréchet differentiable by \cref{thm:objective-deri} and $G$ is $C^2$ by \cref{ass:cost_func}, we can use a Taylor expansion to write
    \begin{multline} \label{eq:obj-diff}
        \begin{aligned}[t]
            j(u) - j(\bar u) 
            &=  G(y_u) - G(\bar y) + \frac{\nu}{2} \left( \norm{u}_{L^2(\Omega)}^2 - \norm{\bar u}_{L^2(\Omega)}^2 \right) \\
            & = G'(\bar y)(y_u - \bar y) +\nu \int_\Omega  \left( u - \bar u \right)  \bar u\dx  \\
            \MoveEqLeft[-1]+ \int_0^1 (1-s)G''(\bar y + s(y_u -\bar y))(y_u - \bar y)^2 \ds + \frac{\nu}{2} \norm{u - \bar u}_{L^2(\Omega)}^2\\
            & = \int_\Omega  (b + a(\bar y )) \nabla \bar \varphi \cdot (\nabla y_u - \nabla \bar y)  + \1_{\{\bar y \notin E_{a} \}} a'(\bar y) \nabla \bar y \cdot \nabla \bar \varphi (y_u - \bar y) \dx\\
            \MoveEqLeft[-1] - \int_\Omega \bar\varphi(u - \bar u) \dx  
            + \int_\Omega \bar d (u - \bar u) \dx \\
            \MoveEqLeft[-1] +  \int_0^1 (1-s)G''(\bar y + s(y_u -\bar y))(y_u - \bar y)^2 \ds + \frac{\nu}{2} \norm{u - \bar u}_{L^2(\Omega)}^2,
        \end{aligned}
    \end{multline}
    where we have employed \eqref{eq:adjoint_OS} and the definition of $\bar d$ to obtain the last equality. Testing the state equations corresponding to $y_u$ and $\bar y$ by $\bar \varphi$ and then subtracting yields
    \begin{equation*}
        \begin{aligned}
            \int_\Omega \bar \varphi (u - \bar u) \dx = \int_\Omega \left[ (b + a(y_u) )(\nabla y_u - \nabla \bar y) +( a(y_u) - a(\bar y) ) \nabla \bar y \right] \cdot \nabla \bar\varphi \dx.
        \end{aligned}
    \end{equation*}
    Inserting this equality into \eqref{eq:obj-diff}, we arrive at the desired conclusion.
\end{proof}

\bigskip

A crucial step of our analysis will be to bound $Q_2(\bar u,\bar y,\bar \varphi; h)$ purely in terms of the jumps of the derivatives of $a$ in the optimal state $\bar y$. To do this, we define the jump functional $\Sigma: W^{1,1}(\Omega) \to [0, \infty]$ via
\begin{equation}
    \label{eq:E-functional}
    \begin{aligned}[t]
        \Sigma(y) & := \limsup_{r \to 0^+} \frac{1}{r} \sum_{m= 1}^{N} \sum_{i=1}^{\barK} \sigma_i \int_\Omega \left[ \1_{\{ 0 < |y - t_i | \leq r \}} \left| \partial_{x_m} y \right| \right]\dx \\
        & = \limsup_{r \to 0^+} \frac{1}{r} \sum_{m= 1}^{N} \sum_{ i\in I_y^+ }   \sigma_i \int_\Omega \left[ \1_{\{ |y - t_i | \leq r \}} \left| \partial_{x_m} y \right| \right]\dx,
    \end{aligned}
\end{equation}
where the (non-negative) $\{\sigma_i\}_{1\leq i \leq \barK}$ are as defined in \eqref{eq:sigma-i} and
\begin{equation*}
    I_y^+:= \left\{i \in \{1,\ldots,\barK\}\,\middle|\, \min\nolimits_{x \in \overline\Omega} y(x) \leq t_i \leq  \max\nolimits_{x \in \overline\Omega} y(x) \right\}.
\end{equation*}
(Note that in contrast to $I_y$ defined in \eqref{eq:index-set-reduced}, we exclude the largest value of $i$ for which $t_i < \min_{x\in\overline \Omega} y(x)$ but include the largest value of $i$ for which $t_i = \max_{x\in\overline\Omega} y(x)$.)

Clearly, if $a'$ is finitely $PC^1$, then $\Sigma(y) = 0$. On the other hand, if $a'$ has points of discontinuity, then, intuitively, $\Sigma(y)$ measures the oscillation of $y\in W^{1,1}(\Omega)$ around these discontinuities (which maybe unbounded even if there are only finitely many such points).
This is illustrated by the following one-dimensional example, whose derivation by straight-forward calculation is given in \cref{app:examples}.

\begin{example} \label{exam:Ey-ocsillation}
    Let $\Omega:= (\alpha_0, \beta_0)\subset\R$ be bounded and let $y: \overline \Omega \to \R$ be a Lipschitz continuous function that is piecewise monotone (increasing or decreasing). Assume that $a$ satisfies \cref{ass:PC1-func} such that $\min_{1 \leq i \leq \barK} \sigma_i>0$ and that 
    \begin{equation*}
        y^{-1}(E_a) := \{x \in [\alpha_0, \beta_0] \mid y(x) \in E_a \} = \bigcup_{j \in J} [\underline{x}_j,\overline x_j]
    \end{equation*} 
    for some $\underline x_j,\overline x_j\subset [\alpha_0,\beta_0]$, $j\in J$, with $[\underline{x}_j,\overline x_j] \cap [\underline{x}_k,\overline x_k] = \emptyset$ for all $j,k\in J$, $j \neq k$. Then, the following assertions hold:
    \begin{enumerate}[label =(\roman*)]
        \item if $\card(J) < \infty$, then $(\card(J) -1) \min_{1 \leq i \leq \barK} \sigma_i \leq \Sigma(y) \leq 2\card(J)\max_{1 \leq i \leq \barK} \sigma_i$;
        \item if $\card(J) = \infty$, then $\Sigma(y) = \infty$;
    \end{enumerate}
    where $\card(J)$ stands for the cardinality of set $J$.   
    In particular, $\Sigma(y)$ is finite if and only if the index set $J$ is finite.
\end{example}
Note that if $\underline x_j < \overline x_j$, then $[\underline x_j, \overline x_j] \subset\{y \in E_a\}$, which therefore has positive Lebesgue measure. This demonstrates that the assumption $\Sigma(y)<\infty$, which will be crucial throughout the following, does not imply that $\{y\in E_a\}$ has measure zero. 

We now prove some technical results on $\Sigma$ and $Q_2$ that will be needed in the following \cref{sec:2nd-OC:soc}. To keep the notation concise, from now on we will simply write $\zeta_i(s,h):=\zeta_i(\bar u, \bar y; s,h)$ and $A_n(\{s_n\},\{h_n\}):=A_n(\bar u,\bar y; \{s_n\},\{h_n\})$.
\begin{lemma}
    \label{lem:invariant}
    Let $\{s_n\} \in c_0^+$ be arbitrary and let $\{h_n\}, \{v_n\} \subset L^2(\Omega)$ be such that $h_n \rightharpoonup h$ and $v_n \rightharpoonup h$ in $L^2(\Omega)$ for some $h \in L^2(\Omega)$. If $\Sigma(\bar y) < \infty$ and $\bar \varphi \in W^{1,\infty}(\Omega)$, then 
    \begin{equation*}
        \lim_{n \to \infty} \frac{1}{s_n^2} \int_\Omega \left[ A_n(\{s_n\}, \{h_n\}) - A_n(\{s_n\}, \{v_n\}) \right] \nabla \bar y \cdot \nabla \bar \varphi \dx = 0.
    \end{equation*}
\end{lemma}
\begin{proof}
    Setting $y_n := S(\bar u + s_n h_n)$, $w_n := S(\bar u + s_n v_n)$ and exploiting the definition \eqref{eq:A-n} yields
    \begin{equation}
        \label{eq:invar-auxi1}
        A_n(\{s_n\}, \{h_n\}) - A_n(\{s_n\}, \{v_n\}) 
        =  \sum_{i =0}^{\barK } \left\{[a'_{i-1}(t_i) - a'_i(t_i)] M_{n,i} +[a'_{i+1}(t_{i+1}) - a'_i(t_{i+1})] L_{n,i} \right\}
    \end{equation}
    with 
    \begin{equation*}
        M_{n,i} := \1_{\Omega_{y_n, \bar y }^{i,2}}(t_i-y_n) - \1_{\Omega_{w_n, \bar y }^{i,2}}(t_i -w_n), \quad L_{n,i} := \1_{\Omega_{y_n, \bar y }^{i,3}}(t_{i+1}-y_n) - \1_{\Omega_{w_n, \bar y }^{i,3}}(t_{i+1}-w_n).
    \end{equation*}
    Setting $r_n := \norm{y_n-w_n}_{L^\infty(\Omega)}$ and $\kappa_n := \max\{\norm{y_n- \bar y }_{L^\infty(\Omega)}, \norm{w_n- \bar y }_{L^\infty(\Omega)} \}$ gives 
    \begin{equation}
        \label{eq:invar-limits}
        \kappa_n  \leq Cs_n \quad \text{and} \quad \frac{r_n}{s_n} \to 0
    \end{equation}
    for some positive constant $C$ due to the differentiability of $S$ and \cref{cor:diff}, respectively.  We can thus assume that $0 \leq \kappa_n, r_n < \delta$ for all $n \in\N$ large enough. Writing 
    \begin{equation*}
        M_{n,i} := - \1_{\Omega_{y_n, \bar y }^{i,2}}(y_n- w_n) - \left(\1_{ \Omega_{{y_n}, \bar y }^{i,2} } -  \1_{\Omega_{w_n, \bar y }^{i,2}} \right) (w_n- t_i) 
    \end{equation*}
    and using the fact that
    \begin{equation*}
        \Omega_{y_n, \bar y }^{i,2} := \left\{ \bar y \in (t_i, t_i +\delta), y_n \in (t_i-\delta, t_i] \right\} \subset \left\{ 0 < \bar y - t_i \leq \kappa_n \right\},
    \end{equation*}
    we derive
    \begin{equation*}
        |M_{n,i}| \leq \1_{\left\{ 0 < \bar y - t_i \leq \kappa_n \right\} } r_n + \left|\left( \1_{ \Omega_{{y_n}, \bar y }^{i,2} } -  \1_{\Omega_{w_n, \bar y }^{i,2}} \right)(w_n- t_i) \right|.
    \end{equation*}
    On the other hand, by a simple calculation, it holds that
    \begin{multline*}
        \left| \left(\1_{ \Omega_{{y_n}, \bar y }^{i,2} } -  \1_{\Omega_{w_n, \bar y }^{i,2}} \right) (w_n- t_i) \right| \\
        \begin{aligned}
            &=  \1_{ \left\{ \bar y \in (t_i, t_i + \delta )\right\}}  \left| \1_{ \{ y_n \in (t_i - \delta, t_i] \} } -  \1_{\{ w_n \in (t_i - \delta, t_i] \} } \right|  |w_n- t_i| \\
            & = \1_{ \left\{ \bar y \in (t_i, t_i + \delta )\right\}} \left| \1_{ \{ y_n \in (t_i - \delta, t_i], w_n \in (t_i,t_i + r_n] \} } -  \1_{\{ w_n \in (t_i - \delta, t_i], y_n \in (t_i, t_i+r_n] \} } \right|  |w_n- t_i|\\
            & \leq \1_{ \left\{ 0 < \bar y - t_i \leq \kappa_n \right\}} r_n.
        \end{aligned}
    \end{multline*}
    Here we have exploited the facts that
    \begin{align*}
        & \left\{ \bar y \in (t_i, t_i + \delta ), y_n \in (t_i - \delta, t_i] \right\} \cup \left\{ \bar y \in (t_i, t_i + \delta ), w_n \in (t_i - \delta, t_i] \right\} \subset \left\{ 0 < \bar y - t_i \leq \kappa_n \right\}
    \end{align*}
    and $|w_n - t_i| \leq |w_n - y_n| \leq r_n$ almost everywhere on $\{ w_n \in (t_i - \delta, t_i], y_n \in (t_i, t_i+r_n] \}$. We therefore have
    \begin{align}
        \label{eq:invar-auxi2}
        |M_{n,i}| &\leq  2r_n\1_{\left\{ 0 < \bar y - t_i \leq \kappa_n \right\} }.
        \intertext{Similarly, there holds that}
        \label{eq:invar-auxi3}
        |L_{n,i}| &\leq 2r_n \1_{\left\{ 0 < t_{i+1} - \bar y \leq \kappa_n \right\} }.
    \end{align}
    Inserting \eqref{eq:invar-auxi2} and \eqref{eq:invar-auxi3} into \eqref{eq:invar-auxi1} and exploiting the obtained  result as well as \eqref{eq:sigma-i} yields 
    \begin{multline*}
        \frac{1}{s_n^2} \int_\Omega \left| \left[ A_n(\{s_n\}, \{h_n\}) - A_n(\{s_n\}, \{v_n\}) \right] \nabla \bar y \cdot \nabla \bar \varphi  \right|\dx  \\
        \begin{aligned}
            &  \leq \frac{2r_n}{s_n^2} \int_\Omega \sum_{i =0}^{\barK } \left[ \sigma_{i} \1_{\left\{ 0 < \bar y - t_i \leq \kappa_n \right\} } + \sigma_{i+1} \1_{\left\{ 0 < t_{i+1} - \bar y \leq \kappa_n \right\} } \right] \left| \nabla \bar y \cdot \nabla \bar\varphi \right| \dx \\
            & = \frac{2r_n \kappa_n}{s_n^2} \times \frac{1}{\kappa_n} \int_\Omega \sum_{i =1}^{\barK } \sigma_{i} \1_{\left\{ 0 < |\bar y - t_i| \leq \kappa_n \right\} }\left| \nabla \bar y \cdot \nabla \bar\varphi \right| \dx.
        \end{aligned}
    \end{multline*}
    Here we have used the fact that $\1_{\left\{ 0 < \bar y - t_{0} \leq \kappa_n \right\} } \equiv \1_{\left\{ 0 < t_{\barK+1} - \bar y \leq \kappa_n \right\} } \equiv 0$.
    Letting $n \to \infty$ and using the definition \eqref{eq:E-functional}, we arrive at 
    \begin{equation*}
        \lim_{n \to \infty} \frac{1}{s_n^2} \int_\Omega \left| \left[ A_n(\{s_n\}, \{h_n\}) - A_n(\{s_n\}, \{v_n\}) \right] \nabla \bar y \cdot \nabla \bar \varphi\right|  \dx  \leq \norm{\nabla \bar\varphi}_{L^\infty(\Omega)}\Sigma(\bar y)  \lim_{n \to \infty}\frac{2r_n \kappa_n}{s_n^2},
    \end{equation*}
    which, together with \eqref{eq:invar-limits}, completes the proof.
\end{proof}
Combining this with \eqref{eq:key-term-sn} and \eqref{eq:key-term}, we obtain the following estimate.
\begin{corollary}\label{cor:sigma-term-bound}
    If $\Sigma(\bar y) < \infty$ and $\bar \varphi \in W^{1,\infty}(\Omega)$, then 
    for any $\{s_n\} \in c_0^+$ and any $h_n \rightharpoonup h$ in $L^2(\Omega)$,
    \begin{equation*}
        \liminf_{n \to \infty} \frac{1}{s_n^2} \int_\Omega A_n(\{s_n\}, \{h_n\}) \nabla \bar y \cdot \nabla \bar \varphi \dx  = \tilde{Q}(\bar u,\bar y,\bar \varphi; \{s_n\},h) \geq Q_2(\bar u,\bar y,\bar \varphi; h).
    \end{equation*}
\end{corollary}
We can use this result to show weak lower semi-continuity of $Q_2$.
\begin{proposition}\label{prop:sigma-wlsc}
    If $\Sigma(\bar y) < \infty$ and $\bar \varphi \in W^{1,\infty}(\Omega)$, then 
    for any $h_n \rightharpoonup h$ in $L^2(\Omega)$,
    \begin{equation*}
        Q_2(\bar u,\bar y,\bar\varphi;h) \leq \liminf_{n \to \infty} Q_2(\bar u,\bar y,\bar\varphi;h_n) .
    \end{equation*}
\end{proposition}
\begin{proof}
    Let $\{h_n\} \subset L^2(\Omega)$ be arbitrary such that $h_n \rightharpoonup h$ in $L^2(\Omega)$. Fixing $n \in\N$ and using the definition \eqref{eq:key-term} shows that there exists a sequence $\{s_j^k(h_n)\}_{j, k \in \N} \in c_0^+$ such that 
    \begin{equation} \label{eq:limit-wlsc}
        s_j^k(h_n) \to 0^+ \text{ as } j \to \infty \quad \text{for all } k \in\N
    \end{equation}
    and 
    \begin{equation*}
        Q_2(\bar u,\bar y,\bar\varphi;h_n) = \lim_{k \to \infty} \tilde{Q}(\bar u,\bar y,\bar\varphi;\{s_j^k(h_n)\}_{j \in \N} , h_n).
    \end{equation*}
    There thus exists a $k_n \geq n$ satisfying 
    \begin{equation}
        \label{eq:weakly-lsc-1}
        Q_2(\bar u,\bar y,\bar\varphi;h_n) - \tilde{Q}(\bar u,\bar y,\bar\varphi;\{s_j^{k_n}(h_n)\}_{j \in \N} , h_n) > - \frac{1}{n}.
    \end{equation}
    The limit in \eqref{eq:limit-wlsc} leads to the existence of a $j_n \in \N$ such that
    \begin{equation}
        \label{eq:lwsc-zero-sequence}
        0 < s_{j}^{k_n}(h_n) < \frac{1}{n} \quad \text{for all } j \geq j_n.
    \end{equation}
    Furthermore, from \eqref{eq:key-term-sn} and \eqref{eq:A-n}, a subsequence $\{s_{j_q}^{k_n}(h_n)\}_{q \in \N}$ of $\{s_j^{k_n}(h_n)\}_{j \in \N}$ exists satisfying
    \begin{equation*}
        \tilde{Q}(\bar u,\bar y,\bar\varphi;\{s_j^{k_n}(h_n)\}_{j \in \N} , h_n) = \lim_{q \to \infty} \frac{1}{(s_{j_q}^{k_n}(h_n))^2} \int_\Omega \sum_{i =0}^{\barK } \zeta_i(s_{j_q}^{k_n}(h_n), h_n)  \nabla \bar y \cdot \nabla \bar \varphi \dx. 
    \end{equation*}
    Then there is a $q_n \in \N$ satisfying $j_{q_n} \geq j_n$ and
    \begin{equation}
        \label{eq:weakly-lsc-2}
        \tilde{Q}(\bar u,\bar y,\bar\varphi;\{s_j^{k_n}(h_n)\}_{j \in \N} , h_n) - \frac{1}{r_n^2} \int_\Omega \sum_{i =0}^{\barK }  \zeta_i(r_n, h_n)  \nabla \bar y \cdot \nabla \bar \varphi \dx > - \frac{1}{n}
    \end{equation}
    with $r_n := s_{j_{q_n}}^{k_n}(h_n)$. By \eqref{eq:lwsc-zero-sequence}, we have $r_n \to 0^+$ as $n \to \infty$ and so $\{r_n\}_{n \in \N} \in c_0^+$. 
    On the other hand, adding \eqref{eq:weakly-lsc-1} and \eqref{eq:weakly-lsc-2} yields that
    \begin{equation*}
        Q_2(\bar u,\bar y,\bar\varphi;h_n)  - \frac{1}{r_n^2} \int_\Omega \sum_{i =0}^{\barK } \zeta_i(r_n, h_n)  \nabla \bar y \cdot \nabla \bar \varphi \dx  > -\frac{2}{n}.
    \end{equation*}
    Taking the limit inferior then shows that 
    \begin{equation*} 
        \begin{aligned}
            \liminf _{n \to \infty} Q_2(\bar u,\bar y,\bar\varphi;h_n) & \geq \liminf _{n \to \infty}\frac{1}{r_n^2} \int_\Omega \sum_{i =0}^{\barK } \zeta_i(r_n, h_n)  \nabla \bar y \cdot \nabla \bar \varphi \dx \\
            & = \liminf _{n \to \infty}\frac{1}{r_n^2} \int_\Omega A_n(\{r_n\}, \{h_n\}) \nabla \bar y \cdot \nabla \bar \varphi \dx,
        \end{aligned}
    \end{equation*}
    where we have used the definition \eqref{eq:A-n} to obtain the last identity. Together with \cref{cor:sigma-term-bound}, this yields the claim.
\end{proof}

\begin{lemma}
    \label{prop:sigma-term}
    Assume that $\Sigma(\bar y) < \infty$ and $\bar \varphi \in W^{1,\infty}(\Omega)$, then for any $h \in L^2(\Omega)$,     
    \begin{equation*}
        \left| Q_2(\bar u,\bar y,\bar\varphi;h) \right| \leq \Sigma(\bar y) \norm{\nabla \bar\varphi}_{L^\infty(\Omega)} \norm{S'(\bar u)h}_{L^\infty(\Omega)}^2.
    \end{equation*}    
\end{lemma}
\begin{proof}
    It suffices to show for any $\{s_n\} \in c_0^+$ that
    \begin{equation}
        \label{eq:sigma-bound}
        \left| \tilde Q(\bar u,\bar y,\bar\varphi;\{s_n\},h) \right| \leq \Sigma(\bar y) \norm{\nabla \bar\varphi}_{L^\infty(\Omega)} \norm{S'(\bar u)h}_{L^\infty(\Omega)}^2.
    \end{equation}
    To this end, we first set $y_n := S(\bar u + s_n h)$ and take $p \in (N,6)$ arbitrary. By the fact that $s_n \to 0^+$ and $L^2(\Omega) \Subset W^{-1,p}(\Omega)$, we have $\bar u + s_nh \to \bar u$ in $W^{-1,p}(\Omega)$ and thus $y_n \to \bar y$ in $W^{1,p}_0(\Omega)$. From this and the embedding $W^{1,p}_0(\Omega) \hookrightarrow C(\overline\Omega)$, it holds that
    \begin{equation*}
        \kappa_n := \norm{y_n - \bar y}_{C(\overline\Omega)} \to 0.
    \end{equation*}
    We can thus assume that $\kappa_n < \delta$ for all $n \in\N$ large enough. 
    Moreover, from \eqref{eq:A-n}, there holds
    \begin{equation}
        \label{eq:A-n-auxi1}
        A_n(\{s_n\},\{h\}) := \sum_{i =0}^{\barK } \zeta_i(s_n,h)  =\sum_{i \in I_{\bar y}} \zeta_i(s_n,h).
    \end{equation}
    We see from \eqref{eq:sigma-i} and the definition of $\zeta_i(s_n,h)$, $\Omega_{y_n, \bar y }^{i,2}$ and $\Omega_{y_n, \bar y }^{i,3}$ that
    \begin{align*}
        \sum_{i \in I_{\bar y}} \left| \zeta_i(s_n,h) \right|  &\leq  \sum_{i \in I_{\bar y} } \left[  \1_{ \{ \bar y \in (t_{i}, t_{i}+\delta ), y_n \in (t_{i}-\delta, t_{i}] \} } \sigma_{i} + \1_{ \{ \bar y \in (t_{i+1}- \delta, t_{i+1}), y_n \in [t_{i+1}, t_{i+1} + \delta) \} }\sigma_{i+1}  \right] |y_n- \bar y| \\
        & \leq \sum_{i \in I_{\bar y} } \left[ \1_{ \{ 0 < \bar y - t_i \leq \norm{y_n - \bar y}_{C(\overline\Omega)} \} } \sigma_{i} + \1_{ \{ 0 < t_{i+1} - \bar y \leq \norm{y_n - \bar y}_{C(\overline\Omega)} \} }\sigma_{i+1}  \right] \norm{y_n - \bar y}_{C(\overline\Omega)}\\
        & = \sum_{i =1}^{\barK} \1_{ \{ 0 < |\bar y - t_i| \leq \kappa_n \} } \sigma_{i} \norm{y_n - \bar y}_{C(\overline\Omega)}.
    \end{align*}
    Consequently, it holds that
    \begin{equation*}
        \frac{1}{s_n^2} \int_\Omega \sum_{i \in I_{\bar y}}  \left| \zeta_i(s_n,h) \nabla \bar\varphi \cdot \nabla \bar y \right|\dx 
        \leq \norm{\nabla\bar \varphi}_{L^\infty(\Omega)} \frac{\norm{y_n - \bar y}_{C(\overline\Omega)}^2}{s_n^2}   \sum_{m= 1}^N \sum_{i =1}^{\barK} \sigma_i   \frac{1}{\kappa_n} \int_\Omega \1_{ \{ 0 < |\bar y - t_i| \leq \kappa_n \} } \left|  \partial_{x_m} \bar y \right|\dx.
    \end{equation*}
    Passing to the limit, employing \cref{cor:diff}, and using \eqref{eq:E-functional} then yields that
    \begin{equation*}
        \limsup_{n \to \infty} \frac{1}{s_n^2} \int_\Omega \sum_{i \in I_{\bar y}} \left| \zeta_i(s_n,h)  \nabla \bar\varphi \cdot \nabla \bar y \right|\dx \leq \norm{\nabla\bar \varphi}_{L^\infty(\Omega)}\norm{S'(\bar u)h}_{L^\infty(\Omega)}^2 \Sigma(\bar y).
    \end{equation*}
    The combination of this with \eqref{eq:A-n-auxi1} gives \eqref{eq:sigma-bound} and thus the claim.
\end{proof}

The following is the main result of this subsection. 
\begin{proposition}
    \label{lem:limits}
    Let \crefrange{ass:domain}{ass:cost_func} hold. Assume further that $G'(\bar y) \in L^{\bar p}(\Omega)$ for some $\bar p >N$. Let $p \in (N, 6)$ be arbitrary and let  $\{s_n\} \in c_0^+$ and $\{h_n\} \subset  L^2(\Omega)$  be arbitrary such that $h_n \rightharpoonup h$ in $L^2(\Omega)$ for some $h \in L^2(\Omega)$.
    Then the following limits hold:
    \begin{enumerate}[label=(\roman*)]
        \item $\frac{1}{s_n^2}\int_0^1 (1-s)G''(\bar y + s(y_n - \bar y))(y_n - \bar y)^2 \ds \to \frac{1}{2}G''(\bar y)(S'(\bar u)h)^2$  with $y_n := S(\bar u + s_n h_n)$;
        \item $\frac{1}{s_n^2} \int_\Omega \left[a(y_n) - a(\bar y) \right] \nabla \bar \varphi \cdot \left( \nabla y_n - \nabla \bar y \right)\dx \to \int_\Omega a'(\bar y; S'(\bar u)h) \nabla\bar \varphi \cdot \nabla (S'(\bar u)h)\dx$;
        \item if, in addition, $\Sigma(\bar y) < \infty$, then 
            \begin{equation*}
                H_n := \int_\Omega [ a(y_n) - a(\bar y) - \1_{\{\bar y \notin E_{a} \} }a'(\bar y)(y_n - \bar y) ] \nabla \bar \varphi \cdot \nabla \bar y \dx, \quad n \in\N,
            \end{equation*} 
            satisfy
            \begin{equation*}
                \limsup_{n \to \infty} \frac{1}{s_n^2} H_n = \frac{1}{2} \int_\Omega \1_{\{\bar y\notin E_{a} \}}a''(\bar y)(S'(\bar u)h)^2 \nabla \bar \varphi \cdot \nabla \bar y \dx - \tilde{Q}(\bar u,\bar y,\bar \varphi; \{s_n\},h).
            \end{equation*}
    \end{enumerate}
\end{proposition}
\begin{proof}
    \emph{(i):} By \cref{cor:diff}, we have
    \begin{equation}
        \label{eq:deri-limit}
        \frac{y_n - \bar y}{s_n} \to S'(\bar u)h \quad \text{in } W^{1,p}_0(\Omega).
    \end{equation}
    This and the dominated convergence theorem give assertion (i).

    \emph{(ii):}  According to \cref{lem:adjoint-equation} and the fact that $G'(\bar y) \in L^{\bar p}(\Omega)$ and $\bar u \in L^{\infty}(\Omega)$, the adjoint state $\bar \varphi$ belongs to $W^{1,\infty}(\Omega)$. Moreover, from \eqref{eq:deri-limit} we have that
    \begin{equation*}
        \frac{1}{s_n} \nabla (y_n - \bar y) \to \nabla S'(\bar u)h \quad\text{in } L^p(\Omega)^N\qquad\text{and}\qquad \frac{1}{s_n} (y_n - \bar y) \to S'(\bar u)h \quad\text{in }C(\overline\Omega).  
    \end{equation*}
    Finally, for all $x\in \overline\Omega$, it holds that $\frac{1}{s_n}\left[a(y_n(x)) - a(\bar y(x)) \right] \to a'(\bar y(x); (S'(\bar u)h)(x))$; see, e.g., \cite[Lem.~3.5]{ClasonNhuRosch}. Therefore, we obtain (ii) from the dominated convergence theorem.

    \emph{(iii): } According to \cref{lem:Lip} and the continuous embedding $W^{1,p}_0(\Omega) \hookrightarrow C(\overline\Omega)$, we obtain $y_n \to \bar y$ in $C(\overline\Omega)$, and we can thus assume that $\norm{y_n - \bar y}_{C(\overline\Omega)} < \delta$ for all $n \in\N$ large enough. Since $\nabla \bar y$ vanishes almost everywhere on $\{\bar y \in E_{a} \}$ \cite[Rem.~2.6]{Chipot2009} and there exists a constant $M>0$ such that $\max\{|y_n(x)|, |\bar y(x)| \} \leq M$ for all $x \in \overline\Omega$, we can write
    \begin{equation}
        \label{eq:Hn-rewrittem}
        \begin{aligned}[t]
            H_n & =\int_\Omega \1_{\{\bar y \notin E_{a} \} } [ a(y_n) - a(\bar y) - a'(\bar y)(y_n - \bar y) ] \nabla \bar \varphi \cdot \nabla \bar y \dx  \\
            & = \int_\Omega T_{y_n, \bar y}  \nabla \bar \varphi \cdot \nabla \bar y \dx  
            = \sum_{i \in I_{\bar y}} \int_\Omega \left[T_{y_n, \bar y}^{i,1} + T_{y_n, \bar y}^{i,2} + T_{y_n, \bar y}^{i,3} \right]\nabla \bar \varphi \cdot \nabla \bar y \dx, 
        \end{aligned}
    \end{equation}
    where $T_{y_n, \bar y}$ and $T_{y_n,\bar y}^{i,j}$, $j = 1,2,3$, are defined in \eqref{eq:T-func} and  \cref{lem:T-decomposition}. 

    We now estimate $T_{y_n, \bar y}^{i,j}$ for $i \in I_{\bar y}$ and $ j = 1,2,3$. Let us fix $i \in I_{\bar y}$ and consider $T_{y_n, \bar y}^{i,1}$. We have
    \begin{equation}
        \label{eq:T-i1}
        T_{y_n, \bar y}^{i,1} = g^{i,1}_n + g^{i,2}_n + g^{i,3}_n
    \end{equation}
    with
    \begin{align*}
        & g^{i,1}_n := \1_{\Omega_{\bar y, i, i +1}^{[\delta, -\delta]} } \left[a_i(y_n) - a_i(\bar y) - a_i'(\bar y)(y_n-\bar y) \right],\\
        & g^{i,2}_n := \1_{\Omega_{\bar y, i, i }^{(0, \delta)} \cap \Omega_{y_n, i, i }^{(0, 2\delta)} } \left[a_i(y_n) - a_i(\bar y) - a_i'(\bar y)(y_n-\bar y) \right],\\
        & g^{i,3}_n := \1_{\Omega_{\bar y, i+1, i+1}^{(-\delta, 0)} \cap \Omega_{y_n, i+1, i+1 }^{(-2\delta, 0)} } \left[a_i(y_n) - a_i(\bar y) - a_i'(\bar y)(y_n-\bar y) \right].
    \end{align*}
    A standard argument shows that 
    \begin{equation}
        \label{eq:g1n-lim}
        \frac{1}{s_n^2} \int_\Omega g^{i,1}_n  \nabla \bar \varphi \cdot \nabla \bar y \dx \to \frac{1}{2} \int_\Omega \1_{\Omega_{\bar y, i, i +1}^{[\delta, -\delta]} } a_i''(\bar y)(S'(\bar u)h)^2 \nabla \bar \varphi \cdot \nabla \bar y \dx.
    \end{equation}
    Since $\norm{y_n - \bar y}_{C(\overline\Omega)} < \delta$ and $\norm{y_n - \bar y}_{C(\overline\Omega)} \to 0$ as $n\to\infty$, there holds
    \begin{align*}
        \1_{\Omega_{\bar y, i, i }^{(0, \delta)} \cap \Omega_{y_n, i, i }^{(0, 2\delta)} } - \1_{\Omega_{\bar y, i, i }^{(0, \delta)}} & = - \1_{\Omega_{\bar y, i, i }^{(0, \delta)} \cap \Omega_{y_n, i, i }^{(-\delta, 0]} } \to 0 \quad \text{as } n \to \infty
    \end{align*}
    almost everywhere in $\Omega$, which together with the dominated convergence theorem yields 
    \begin{equation}
        \label{eq:g2n-lim}
        \frac{1}{s_n^2} \int_\Omega g^{i,2}_n  \nabla \bar \varphi \cdot \nabla \bar y \dx \to \frac{1}{2} \int_\Omega \1_{\Omega_{\bar y, i, i}^{(0,\delta)} } a_i''(\bar y)(S'(\bar u)h)^2 \nabla \bar \varphi \cdot \nabla \bar y \dx.
    \end{equation}
    Similarly, it holds that
    \begin{equation}
        \label{eq:g3n-lim}
        \frac{1}{s_n^2} \int_\Omega g^{i,3}_n  \nabla \bar \varphi \cdot \nabla \bar y \dx \to \frac{1}{2} \int_\Omega \1_{\Omega_{\bar y, i+1, i+1}^{(-\delta,0)} } a_i''(\bar y)(S'(\bar u)h)^2 \nabla \bar \varphi \cdot \nabla \bar y \dx.
    \end{equation}
    From \eqref{eq:T-i1}--\eqref{eq:g3n-lim} and the fact that $\Omega_{\bar y, i, i +1}^{[\delta, -\delta]} \cup \Omega_{\bar y, i, i}^{(0,\delta)} \cup  \Omega_{\bar y, i+1, i+1}^{(-\delta,0)} = \Omega_{\bar y, i, i+1}^{(0,0)}  =: \{ \bar y \in (t_i, t_{i+1}) \}$, we deduce that
    \begin{equation*}
        \frac{1}{s_n^2} \int_\Omega T_{y_n, \bar y}^{i,1}  \nabla \bar \varphi \cdot \nabla \bar y \dx \to  \frac{1}{2} \int_\Omega \1_{ \{ \bar y \in (t_i, t_{i+1}) \} } a_i''(\bar y)(S'(\bar u)h)^2 \nabla \bar \varphi \cdot \nabla \bar y \dx,
    \end{equation*} 
    which, together with the fact that $I_{\bar y}$ is finite, implies that
    \begin{equation}
        \label{eq:Ti1-lim}
        \frac{1}{s_n^2} \sum_{i \in I_{\bar y}} \int_\Omega T_{y_n, \bar y}^{i,1}  \nabla \bar \varphi \cdot \nabla \bar y \dx \to  \frac{1}{2} \int_\Omega \1_{ \{ \bar y \notin E_{a} \} } a''(\bar y)(S'(\bar u)h)^2 \nabla \bar \varphi \cdot \nabla \bar y \dx.
    \end{equation}

    We now estimate $T_{y_n, \bar y}^{i,2}$. To this end, we write
    \begin{equation}   \label{eq:T-i2}
        \begin{aligned}[t]
            T_{y_n, \bar y}^{i,2}& = \1_{ \Omega_{y_n, \bar y }^{i,2} }\left[ a_{i-1}(y_n) - a_i(\bar y) - a_i'(\bar y)(y_n -\bar y) \right]   \\
            & = \left\{ \1_{ \Omega_{y_n, \bar y }^{i,2} }\left[ a_{i-1}(y_n) - a_{i-1}(t_i) - a_{i-1}'(t_i)(y_n -t_i) \right]\right. \\
            \MoveEqLeft[-2] \left.+ \1_{ \Omega_{y_n, \bar y }^{i,2} }\left[ a_{i}(t_i) - a_{i}(\bar y) - a_i'(\bar y)(t_i - \bar y) \right]  + \1_{ \Omega_{y_n, \bar y }^{i,2} }\left[ a_{i}'(t_i) - a_i'(\bar y) \right](y_n- t_i) \right\} \\
            \MoveEqLeft[-1] + \1_{ \Omega_{y_n, \bar y }^{i,2} }\left[ a_{i-1}'(t_i) - a_i'(t_i) \right](y_n- t_i) \\
            & =: \hat T_{y_n, \bar y}^{i,2} - \1_{ \Omega_{y_n, \bar y }^{i,2} } \left[ a_{i-1}'(t_i) - a_i'(t_i) \right] (t_i-y_n) 
        \end{aligned}
    \end{equation}
    with $\Omega_{y_n, \bar y }^{i,2} := \Omega_{\bar y, i, i }^{(0, \delta)} \cap \Omega_{y_n, i, i }^{(-\delta, 0]} = \{ \bar y \in (t_i, t_i + \delta), y_n \in (t_i - \delta, t_i] \}$.
    For almost every $x \in \Omega_{y_n, \bar y }^{i,2}$, we have 
    \begin{equation} \label{eq:Omega-i2-set}
        0 \leq \bar y(x) - t_i, t_i - y_n(x) \leq |y_n(x) - \bar y(x)| \leq \norm{y_n - \bar y}_{C(\overline\Omega)} < \delta.
    \end{equation}
    Combined with the fact that $|y_n(x)| \leq M$ for all $x \in \overline\Omega$ and that $a_{i-1}$ is of class $C^2$, we obtain 
    \begin{equation*}
        \begin{aligned}
            \frac{1}{s_n^2}\left|\1_{ \Omega_{y_n, \bar y }^{i,2} }(x)\left[ a_{i-1}(y_n(x)) - a_{i-1}(t_i) - a_{i-1}'(t_i)(y_n(x) -t_i)  \right] \right| & \leq  \frac{1}{2s_n^2}C \1_{ \Omega_{y_n, \bar y }^{i,2} }(x) \norm{y_n - \bar y}_{C(\overline\Omega)}^2 \\
            & \leq \frac{1}{2}C \1_{ \Omega_{y_n, \bar y }^{i,2} }(x) \to 0
        \end{aligned}
    \end{equation*} 
    for almost every $x \in \Omega$ and for some constant $C$. Here we employed the Lipschitz continuity of $S$ (see \cref{lem:Lip}) and the embedding $W^{1,p}_0(\Omega) \hookrightarrow C(\overline\Omega)$ to deduce the last inequality. It therefore holds that
    \begin{equation}
        \label{eq:T-i2-auxi1}
        \frac{1}{s_n^2}\int_\Omega \left|\1_{ \Omega_{y_n, \bar y }^{i,2} }\left[ a_{i-1}(y_n) - a_{i-1}(t_i) - a_{i-1}'(t_i)(y_n -t_i)  \right] \nabla \bar\varphi \cdot \nabla \bar y \right| \dx \to 0.
    \end{equation}
    Similarly, we obtain
    \begin{equation}
        \label{eq:T-i2-auxi2}
        \frac{1}{s_n^2}\int_\Omega \left|\1_{ \Omega_{y_n, \bar y }^{i,2} }\left[ a_{i}(t_i) - a_{i}(\bar y) - a_i'(\bar y)(t_i - \bar y)  \right] \nabla \bar\varphi \cdot \nabla \bar y \right| \dx \to 0
    \end{equation}
    and
    \begin{equation}
        \label{eq:T-i2-auxi3}
        \frac{1}{s_n^2}\int_\Omega \left|\1_{ \Omega_{y_n, \bar y }^{i,2} }\left[ a_{i}'(t_i) - a_i'(\bar y) \right](y_n- t_i) \nabla \bar\varphi \cdot \nabla \bar y \right| \dx \to 0.
    \end{equation}
    Combining \eqref{eq:T-i2-auxi1}--\eqref{eq:T-i2-auxi3} thus gives 
    \begin{equation}
        \label{eq:T-i2-zero-part}
        \frac{1}{s_n^2} \int_\Omega \left| \hat T_{y_n, \bar y}^{i,2} \nabla \bar\varphi \cdot \nabla \bar y \right| \dx \to 0.
    \end{equation}
    By the same argument as for \eqref{eq:T-i2} and \eqref{eq:T-i2-zero-part}, $T_{y_n, \bar y}^{i,3}$ can be written in the form
    \begin{equation} \label{eq:T-i3}
        T_{y_n, \bar y}^{i,3} = \hat T_{y_n, \bar y}^{i,3} - \1_{ \Omega_{y_n, \bar y }^{i,3} }\left[a_{i+1}'(t_{i+1}) - a_{i}'(t_{i+1})\right] (t_{i+1}-y_n)
    \end{equation}
    with $\Omega_{y_n, \bar y }^{i,3} = \{ \bar y \in (t_{i+1}- \delta, t_{i+1}), y_n \in [t_{i+1}, t_{i+1} + \delta) \}$, and $\hat T_{y_n, \bar y}^{i,3}$ satisfying
    \begin{equation}
        \label{eq:T-i3-zero-part}
        \frac{1}{s_n^2} \int_\Omega \left| \hat T_{y_n, \bar y}^{i,3} \nabla \bar\varphi \cdot \nabla \bar y \right| \dx \to 0.
    \end{equation}
    By combining \eqref{eq:Hn-rewrittem} with the limits \eqref{eq:Ti1-lim}, \eqref{eq:T-i2}, \eqref{eq:T-i2-zero-part}, \eqref{eq:T-i3}, \eqref{eq:T-i3-zero-part} and the definition \eqref{eq:A-n}, we can conclude that
    \begin{equation*}
        \limsup_{n \to \infty} \frac{1}{s_n^2} H_n = \frac{1}{2} \int_\Omega \1_{\{\bar y\notin E_{a} \}}a''(\bar y)(S'(\bar u)h)^2 \nabla \bar \varphi \cdot \nabla \bar y \dx 
        - \liminf_{n \to \infty} \frac{1}{s_n^2} \int_\Omega A_n(\{s_n\},\{h_n\}) \nabla \bar y \cdot \nabla \bar\varphi \dx.
    \end{equation*}
    Together with \cref{cor:sigma-term-bound}, this gives (iii).
\end{proof}

\subsection{Second-order conditions}\label{sec:2nd-OC:soc}

We now have everything at hand to prove the following two theorems that are the main results of the paper, providing no-gap second-order necessary and sufficient conditions in terms of the curvature functionals $Q_s$, $Q_1$, and $Q_2$ defined in \cref{sec:2nd-OC:curvature}.

\begin{theorem}[second-order necessary optimality condition] 
    \label{thm:2nd-OS-nec}
    Let \crefrange{ass:domain}{ass:cost_func} hold. Assume that $\bar u$ is a local optimal solution of \eqref{eq:P2} such that $G'(\bar y) \in L^{\bar p}(\Omega)$ and $\Sigma(\bar y) <\infty$ for some $\bar p >N$ and $\bar y:=S(\bar u)$. Then there exists an adjoint state $\bar\varphi \in H^1_0(\Omega) \cap W^{1,\infty}(\Omega)$ that together with $\bar u, \bar y$ satisfies \eqref{eq:1st-OS} as well as
    \begin{equation}
        \label{eq:2nd-OS-nec}
        Q_s(\bar u, \bar y, \bar \varphi;h,h) + Q_1(\bar u, \bar y, \bar \varphi;h,h) + Q_2(\bar u,\bar y,\bar\varphi;h) \geq 0 \qquad\text{for all }h\in \mathcal{C}(\mathcal{U}_{ad};\bar u).
    \end{equation}
\end{theorem}
\begin{proof}
    The existence of a $\bar\varphi \in H^1_0(\Omega)$ satisfying \eqref{eq:1st-OS} follows from \cref{thm:1st-OC}, while the claimed regularity of $\bar\varphi$ follows from \cref{lem:adjoint-equation}. It remains to prove \eqref{eq:2nd-OS-nec}. To this end, let $h \in \mathcal{C}(\mathcal{U}_{ad};\bar u)$ and $\{s_n\} \in c_0^+$ be arbitrary but fixed. We only need to show that
    \begin{equation}
        \label{eq:2nd-OS-nec-2}
        Q_s(\bar u, \bar y, \bar \varphi;h,h)+Q_1(\bar u, \bar y, \bar \varphi;h,h) + \tilde{Q}(\bar u,\bar y,\bar\varphi;\{s_n\},h) \geq 0.
    \end{equation}
    In order to verify \eqref{eq:2nd-OS-nec-2}, we first see from the definition of $\tilde{Q}(\bar u,\bar y,\bar \varphi; \{s_n\},h)$ that there exists a subsequence $\{s_{n_k}\}$ satisfying
    \begin{equation}
        \label{eq:sigma-subsequence}
        \tilde{Q}(\bar u,\bar y,\bar\varphi;\{s_n\},h) =  \lim_{k \to \infty} \frac{1}{s_{n_k}^2} \int_\Omega  A_{n_k} (\{s_{n_k}\}, \{h\}) \nabla \bar y \cdot \nabla \bar \varphi \dx.
    \end{equation}
    Since $\mathcal{U}_{ad}$ is polyhedric, there are sequences $\{h_m\} \subset L^2(\Omega)$, $\{q_m\} \subset \mathcal{U}_{ad}$, and $\{\lambda_m \} \subset (0,\infty)$ such that
    \begin{equation*}
        h_m \to h \quad \text{in } L^2(\Omega), \quad h_m = \frac{q_m - \bar u}{\lambda_m}, \quad \text{and} \quad h_m \in \bar d^\bot \quad \text{for all } m \in\N.
    \end{equation*}
    Since $s_{n_k} \to 0^+$ as $k \to \infty$, a subsequence, denoted by $\{r_m\}$, of $\{s_{n_k}\}$ exists such that $ 0 < r_m \leq \lambda_m$ for all $m \in\N$. This and the convexity of $\mathcal{U}_{ad}$ yield that
    \begin{equation*}
        \bar u + r_m h_m = \left(1 - \frac{r_m}{\lambda_m}  \right)\bar u + \frac{r_m}{\lambda_m} q_m \in \mathcal{U}_{ad} \quad \text{for all } m \in\N.
    \end{equation*}
    Since $\bar u$ is a local minimizer of \eqref{eq:P2}, it holds that 
    \begin{equation*}
        \frac{1}{r_m^2}\left( j(\bar u + r_m h_m) - j(\bar u) \right) \geq 0
    \end{equation*}
    for all $m\in\N$ large enough. Taking the limit inferior and employing the fact that $h_m \in \bar d^\bot$ and $h_m \to h$ in $L^2(\Omega)$, we obtain from \cref{lem:2nd-expression,lem:limits} that
    \begin{equation*}
        Q_s(\bar u, \bar y, \bar \varphi;h,h)+Q_1(\bar u, \bar y, \bar \varphi;h,h)  + \tilde{Q}(\bar u,\bar y,\bar \varphi; \{r_m\},h) \geq 0.
    \end{equation*}
    Moreover, we have $\tilde{Q}(\bar u,\bar y,\bar \varphi; \{r_m\},h) = \tilde{Q}(\bar u,\bar y,\bar \varphi; \{s_n\},h)$ as a result of \eqref{eq:sigma-subsequence} and the fact that $\{r_m\}$ is a subsequence of $\{s_{n_k}\}$.
    This finally gives \eqref{eq:2nd-OS-nec-2}.
\end{proof}

\begin{theorem}[second-order sufficient optimality conditions] 
    \label{thm:2nd-OS-suf}
    Let \crefrange{ass:domain}{ass:cost_func} hold. Assume that $\bar u$ is a feasible point of \eqref{eq:P2} such that $G'(\bar y) \in L^{\bar p}(\Omega)$ and $\Sigma(\bar y) < \infty$ for some $\bar p >N$ and $\bar y:=S(\bar u)$. Assume further that  there exists an adjoint state $\bar\varphi \in H^1_0(\Omega) \cap W^{1,\infty}(\Omega)$ that together with $\bar u, \bar y$ satisfies the first-order optimality conditions \eqref{eq:1st-OS} as well as
    \begin{equation}
        \label{eq:2nd-OS-suff}
        Q_s(\bar u, \bar y, \bar \varphi;h,h) +Q_1(\bar u, \bar y, \bar \varphi;h,h) +  Q_2(\bar u,\bar y,\bar\varphi;h) >0
        \qquad\text{for all }h\in \mathcal{C}(\mathcal{U}_{ad};\bar u)\setminus \{0\}.
    \end{equation}
    Then there exist constants $c, \rho >0$ such that
    \begin{equation}
        \label{eq:quadratic-grownth}
        j(\bar u) + c \norm{u - \bar u}_{L^2(\Omega)}^2 \leq j(u) \quad \text{for all } u \in \mathcal{U}_{ad} \cap \overline B_{L^2(\Omega)}(\bar u, \rho).
    \end{equation}
    In particular, $\bar u$ is a strict local minimizer of \eqref{eq:P2}.
\end{theorem}
\begin{proof}
    We argue by contradiction. Assume that there exists a sequence $\{u_n\} \subset \mathcal{U}_{ad}$ such that
    \begin{equation} \label{eq:contradiction}
        \norm{u_n - \bar u}_{L^2(\Omega)} < \frac{1}{n} \quad \text{and} \quad j(\bar u) + \frac{1}{n}\norm{u_n - \bar u}_{L^2(\Omega)}^2 > j(u_n),\quad n\in\N.
    \end{equation}
    Setting $s_n := \norm{u_n - \bar u}_{L^2(\Omega)}$ and $h_n := \frac{u_n - \bar u}{s_n}$ yields that $\norm{h_n}_{L^2(\Omega)} = 1$. Then there exists a subsequence of $\{h_n\}$, also denoted in the same way, such that $h_n \rightharpoonup h$ in $L^2(\Omega)$ for some $h \in L^2(\Omega)$. 

    We first verify that $h \in \mathcal{C}(\mathcal{U}_{ad};\bar u)$. First, we have that $h_n \in \mathcal{R}(\mathcal{U}_{ad}; \bar u)$ and thus $h_n(x) \geq 0$ if $\bar u(x) = \alpha(x)$ and $h_n(x) \leq 0$ if $\bar u(x) = \beta(x)$ for almost every $x \in \Omega$. From this and $h_n \weakto h$, we deduce that  $h(x) \geq 0$ if $\bar u(x) = \alpha(x)$ and $h(x) \leq 0$ if $\bar u(x) = \beta(x)$ for almost every $x \in \Omega$. Consequently, it holds that $h \in \mathcal{T}(\mathcal{U}_{ad};\bar u)$. 
    Since $j$ is continuously differentiable as a function from $L^2(\Omega)$ to $\R$ according to \cref{thm:objective-deri},
    a Taylor expansion thus gives
    \begin{equation*}
        j(u_n) = j(\bar u) + j'(\bar u)(u_n - \bar u) + o(\norm{u_n - \bar u}_{L^2(\Omega)}).
    \end{equation*}
    This, together with the last inequality in \eqref{eq:contradiction}, implies, for $n$ large enough, that
    \begin{equation*}
        j'(\bar u)(u_n - \bar u) + o(\norm{u_n - \bar u}_{L^2(\Omega)}) < \frac{1}{n} s_n^2.
    \end{equation*}
    Dividing the above inequality by $s_n$ and then passing to the limit, we have $j'(\bar u)h \leq 0$. Furthermore, it follows from \eqref{eq:normal_OS} that $j'(\bar u)v \geq 0$ for all $v \in \mathcal{U}_{ad} - \bar u$ and thus for all $v \in \mathcal{T}(\mathcal{U}_{ad}; \bar u)$. In particular, we have $j'(\bar u)h \geq 0$ and thus $j'(\bar u)h = 0$. Hence, it holds that $h \in \bar d^\bot$ and so $h \in \mathcal{C}(\mathcal{U}_{ad};\bar u)$.

    We now obtain a contradiction and thus complete the proof. Indeed, from the last inequality in \eqref{eq:contradiction}, we obtain
    \begin{equation*}
        \frac{1}{s_n^2} \left[j(u_n) - j(\bar u) \right] < \frac{1}{n}.
    \end{equation*}
    Combining this with \eqref{eq:normal_OS}, \eqref{lem:2nd-expression}, and \cref{lem:2nd-expression} yields that 
    \begin{multline*}
        \frac{1}{n} > \frac{1}{s_n^2} \left \{ \int_0^1 (1-s)G''(\bar y + s(y_n -\bar y))(y_n - \bar y)^2 \ds + \frac{\nu}{2} s_n^2 \norm{h_n}_{L^2(\Omega)}^2 \right. \\
        \begin{aligned}
            & - \int_\Omega \left[a(y_n) - a(\bar y) \right] \nabla \bar \varphi \cdot \left( \nabla y_n - \nabla \bar y \right)\dx\\
            & \left. - \int_\Omega \left[ a(y_n) - a(\bar y) - \1_{\{\bar y \notin E_{a} \} }a'(\bar y)(y_n - \bar y) \right] \nabla \bar \varphi \cdot \nabla \bar y \dx\right \}  
        \end{aligned} 
    \end{multline*}
    with $y_n := S(u_n)$.
    Taking the limit inferior, employing \cref{lem:limits}, and using that $h_n\weakto h$ in $L^2(\Omega)$, we arrive at
    \begin{equation*} 
        0 \geq Q_s(\bar u, \bar y, \bar\varphi;h,h)+Q_1(\bar u, \bar y, \bar\varphi;h,h) + \tilde{Q}(\bar u,\bar y,\bar \varphi; \{s_n\},h) + \frac{\nu}{2}(1 -\norm{h}_{L^2(\Omega)}^2).
    \end{equation*}
    Consequently, 
    \begin{equation}  \label{eq:contradiction2}
        0 \geq Q_s(\bar u, \bar y, \bar\varphi;h,h)+Q_1(\bar u, \bar y, \bar\varphi;h,h) + Q_2(\bar u,\bar y,\bar\varphi;h) + \frac{\nu}{2}(1 -\norm{h}_{L^2(\Omega)}^2).
    \end{equation}
    Since the norm in $L^2(\Omega)$ is weakly lower semicontinuous, there holds $\norm{h}_{L^2(\Omega)} \leq 1$ by definition of $h_n$. 
    From this, \eqref{eq:contradiction2}, and \eqref{eq:2nd-OS-suff}, we have $h = 0$. Inserting $h=0$ into \eqref{eq:contradiction2} and exploiting the fact that $Q_2(\bar u,\bar y,\bar \varphi; 0) = 0$ yields that $0 \geq \frac{\nu}{2} > 0$, which is impossible.
\end{proof}

We finish this section by providing another version of the sufficient second-order optimality conditions that are equivalent to \eqref{eq:2nd-OS-suff} and could be useful for estimating discretization errors in finite element approximations. The proof of the next result is partly based on \cite[Thm.~4.4]{CasasMateos2002} with some modifications.
\begin{proposition}
    \label{thm:2nd-OS-suf2}
    Let \crefrange{ass:domain}{ass:cost_func} hold. Assume that $\bar u$ is a feasible point of \eqref{eq:P2}  such that $G'(\bar y) \in L^{\bar p}(\Omega)$ and $\Sigma(\bar y) < \infty$ for some $\bar p >N$ and $\bar y:=S(\bar u)$. Assume further that  there exists an adjoint state $\bar\varphi \in H^1_0(\Omega) \cap W^{1,\infty}(\Omega)$ that together with $\bar u, \bar y$ satisfies \eqref{eq:1st-OS}. Then, \eqref{eq:2nd-OS-suff} holds if and only if there exist constants $c_0, \tau > 0$ such that
    \begin{equation}
        \label{eq:2nd-OS-suff-equiv}
        Q_s(\bar u, \bar y, \bar \varphi;h,h) +Q_1(\bar u, \bar y, \bar \varphi;h,h) + Q_2(\bar u,\bar y,\bar\varphi;h) \geq  c_0 \norm{h}_{L^2(\Omega)}^2 \qquad\text{for all } h \in \mathcal{C}^\tau(\mathcal{U}_{ad};\bar u)
    \end{equation}
    for
    \begin{equation*}
        \mathcal{C}^\tau(\mathcal{U}_{ad};\bar u) := \left\{ h \in L^2(\Omega) \,\middle|\, h(x) 
            \begin{cases}
                \geq 0 & \text{if }  \bar u(x) = \alpha(x)\\
                \leq 0 & \text{if }  \bar u(x) = \beta(x) \\
                = 0 & \text{if }  |\bar d(x)| > \tau 
            \end{cases} 
            \text{a.e. } x \in \Omega
        \right\}.
    \end{equation*}
\end{proposition}
\begin{proof}
    Since the inclusion $\mathcal{C}(\mathcal{U}_{ad};\bar u) \subset \mathcal{C}^\tau(\mathcal{U}_{ad};\bar u)$ holds, the inequality \eqref{eq:2nd-OS-suff} is a direct consequence of \eqref{eq:2nd-OS-suff-equiv}. We thus only need to prove the implication ``\eqref{eq:2nd-OS-suff} $\Rightarrow$ \eqref{eq:2nd-OS-suff-equiv}''. To this end, assume that \eqref{eq:2nd-OS-suff-equiv} is not fulfilled. Then, for any $n \in\N$, there exists $v_n \in \mathcal{C}^{1/n}(\mathcal{U}_{ad};\bar u)$ such that
    \begin{equation*}
        Q_s(\bar u, \bar y, \bar \varphi;v_n, v_n)+Q_1(\bar u, \bar y, \bar \varphi;v_n, v_n) + Q_2(\bar u,\bar y,\bar\varphi;v_n) < \frac{1}{n} \norm{v_n}_{L^2(\Omega)}^2.
    \end{equation*}
    Dividing this inequality by $\norm{v_n}^2_{L^2(\Omega)}$, using the positive homogeneity of $(Q_s+Q_1)(\bar u, \bar y, \bar \varphi;\cdot,\cdot)$ as well as \cref{lem:homogeneity}, and setting $h_n := \frac{v_n}{\norm{v_n}_{L^2(\Omega)}}$, we have that $h_n \in \mathcal{C}^{1/n}(\mathcal{U}_{ad};\bar u)$, $\norm{h_n}_{L^2(\Omega)} = 1$ and
    \begin{equation}
        \label{eq:contradiction-equi-thm}
        Q_s(\bar u, \bar y, \bar \varphi;h_n, h_n)+Q_1(\bar u, \bar y, \bar \varphi;h_n, h_n) + Q_2(\bar u,\bar y,\bar\varphi;h_n) < \frac{1}{n}.
    \end{equation}
    Since $\{h_n\}$ is bounded in $L^2(\Omega)$, there exists a subsequence, denoted in the same way, such that $h_n \rightharpoonup h$ in $L^2(\Omega)$ for some $h \in L^2(\Omega)$. Obviously, it holds that $h \in \mathcal{C}(\mathcal{U}_{ad};\bar u)$. 
    The compact embedding $L^2(\Omega) \Subset W^{-1,p}(\Omega)$ for any $p \in (N, 6)$ implies that $h_n \to h$ in $W^{-1,p}(\Omega)$. Thus, we have $S'(\bar u)h_n \to S'(\bar u)h$ in $W^{1,p}_0(\Omega)$ and so in $H^1_0(\Omega) \cap C(\overline\Omega)$. From this, the weak lower semicontinuity of the norm in $L^2(\Omega)$, the estimate \eqref{eq:contradiction-equi-thm}, the definition of $Q_s$ and $Q_1$ and the weak lower semicontinuity of $Q_2(\bar u,\bar y,\bar\varphi;\cdot)$ from \cref{prop:sigma-wlsc} we deduce that
    \begin{equation} \label{eq:contradiction2-thm}
        \begin{aligned}[t]
            (Q_s+Q_1)(\bar u, \bar y, \bar \varphi;h, h) + Q_2(\bar u,\bar y,\bar\varphi;h) 
            &  \leq \liminf_{n \to \infty}(Q_s+Q_1)(\bar u, \bar y, \bar \varphi;h_n, h_n)   + \liminf_{n \to \infty} Q_2(\bar u,\bar y,\bar\varphi;h_n) \\
            & \leq \liminf_{n \to \infty} \left[(Q_s+Q_1)(\bar u, \bar y, \bar \varphi;h_n, h_n) + Q_2(\bar u,\bar y,\bar\varphi;h_n) \right] \\
            &\leq  0. 
        \end{aligned}
    \end{equation}
    From this and \eqref{eq:2nd-OS-suff}, it holds that $h = 0$. Again, we see from the definition of $Q_s$ and $Q_1$ and the fact that $\norm{h_n}_{L^2(\Omega)} = 1$ that
    \begin{multline*} 
        (Q_s+Q_1)(\bar u, \bar y, \bar \varphi;h_n, h_n) = \frac{1}{2}G''(\bar y)(S'(\bar u)h_n )^2 + \frac{\nu}{2}  - \int_\Omega a'(\bar y; S'(\bar u)h_n) \nabla (S'(\bar u)h_n) \cdot  \nabla\bar \varphi  \dx   \\
        - \frac{1}{2} \int_\Omega \1_{\{\bar y\notin E_{a} \}}a''(\bar y)(S'(\bar u)h_n)^2  \nabla \bar \varphi \cdot \nabla \bar y \dx.
    \end{multline*} 
    Combining this with the fact that $\nabla (S'(\bar u)h_n) \to 0$ in $L^p(\Omega)^N$ and $S'(\bar u)h_n \to 0$ in $C(\overline\Omega) \cap H^1_0(\Omega)$, we can conclude from the dominated convergence theorem and \cref{ass:cost_func} that
    \begin{equation*}
        \begin{aligned}
            \frac{\nu}{2} = \lim_{n \to \infty}(Q_s+Q_1)(\bar u, \bar y, \bar \varphi;h_n, h_n) 
            &\leq \limsup_{n \to \infty} \left[ - Q_2(\bar u,\bar y,\bar\varphi;h_n) \right] \\
            &\leq \Sigma(\bar y) \norm{\nabla \bar\varphi}_{L^\infty(\Omega)} \lim_{n \to \infty}\norm{S'(\bar u)h_n}_{L^\infty(\Omega)}^2 = 0,
        \end{aligned}
    \end{equation*} 
    where we have used the limit \eqref{eq:contradiction2-thm} and \cref{prop:sigma-term} to obtain the last two estimates.
    This gives the desired contradiction and completes the proof.
\end{proof}

\section{Conclusions}

We have derived second-order optimality conditions for an optimal control problem governed by a quasilinear elliptic differential equation with a finitely $PC^2$ coefficient. 
Showing that the control-to-state operator is in fact Fréchet differentiable
(but in general not twice differentiable) allows using a second-order Taylor-type expansion to formulate necessary and sufficient conditions in terms of a new curvature functional related to the jump of the first derivatives of the non-smooth coeffficients in critical points. These are no-gap conditions in the sense that the only difference between necessary and sufficient conditions lies in the fact that the inequality in the latter is strict.
Furthermore, an equivalent formulation of the second-order sufficient optimality condition that could be used for discretization error estimates is also derived. Such estimates will be studied in a follow-up work.

\appendix

\section{Regularity of a quasilinear equation on convex domains} \label{app:regu-convex}
\begin{lemma}
    \label{lem:regularity-convex}
    Let \crefrange{ass:domain}{ass:PC1-func} hold and $q > N$. If $y_u \in W^{1,2q}_0(\Omega)$ is the unique solution to
    \begin{equation} \label{eq:state-app}
        \left\{
            \begin{aligned}
                -\div [\left(b + a(y_u) \right) \nabla y_u ]& = u && \text{in } \Omega, \\
                y_u &=0 && \text{on } \partial\Omega,
            \end{aligned}
        \right.
    \end{equation} 
    with $u \in L^q(\Omega)$ and satisfies $\norm{y_u}_{C(\overline\Omega)} \leq M$ for some $M >0$, then $y_u \in H^2(\Omega) \cap W^{1,\infty}(\Omega)$ and
    \begin{equation} \label{eq:apriori-app}
        \norm{y_u}_{H^2(\Omega)} + \norm{y_u}_{W^{1,\infty}(\Omega)} \leq C\left(q,M, \norm{u}_{L^q(\Omega)},\norm{y_u}_{W^{1,2q}_0(\Omega)}\right).
    \end{equation}
\end{lemma}
\begin{proof}
    Define the function
    \begin{equation*}
        a^M: \R \to \R,\qquad       a^M(t) := 
        \begin{cases}
            a(t) & \text{if } |t| \leq 2M,\\
            a(2M) & \text{if } t>2M,\\
            a(-2M)& \text{if } t< -2M. 
        \end{cases}
    \end{equation*}
    Obviously, $a^M$ is a $PC^1$-function and $\nabla a^M \in L^\infty(\R)^N$.
    By the chain rule (see, e.g., \cite[Thm.~7.8]{Gilbarg_Trudinger}), it holds that
    \begin{equation*}
        \frac{\partial a(y_u)}{\partial x_i} = \frac{\partial a^M(y_u)}{\partial x_i} = \1_{\{y_u \notin E_{a}\} } a'(y_u) \frac{\partial y_u}{\partial x_i}.
    \end{equation*}
    Moreover, by employing \cref{ass:PC1-func}, we deduce that $\left|\1_{\{y_u \notin E_{a}\} } a'(y_u) \right| \leq C_M$ for almost every $x\in\Omega$ and for some constant $C_M>0$.
    Consider now the equation
    \begin{equation} \label{eq:state-convex}
        \left\{
            \begin{aligned}
                -\Delta \tilde y &= \frac{1}{b+ a(y_u)} \left[ u + \nabla b \cdot \nabla y_u + \1_{\{y_u \notin E_{a} \}} a'(y_u) |\nabla y_u |^2 \right] && \text{in } \Omega, \\
                \tilde y &=0 && \text{on } \partial\Omega.
            \end{aligned}
        \right.
    \end{equation}
    Since the right-hand side of \eqref{eq:state-convex} belongs to $L^q(\Omega)$ for $q >N \geq 2$, it holds that $\tilde{y} \in H^2(\Omega)$ according to \cite[Thm.~3.2.1.2]{Grisvard1985}. Furthermore, we have from the fact that $b(x) \geq \underline{b} >0$ for all $x\in \overline\Omega$ and the non-negativity of $a$ that
    \begin{align*}
        \norm{\Delta \tilde{y}}_{L^q(\Omega)} & \leq \frac{1}{\underline{b}} \left[  \norm{u}_{L^q(\Omega)} + \norm{\nabla b \cdot \nabla y_u}_{L^q(\Omega)} + \norm{\1_{\{y_u \notin E_{a} \}} a'(y_u) |\nabla y_u |^2}_{L^q(\Omega)} \right] \\
        & \leq \frac{1}{\underline{b}}\left[  \norm{u}_{L^q(\Omega)} + L_{b}\norm{\nabla y_u}_{L^q(\Omega)} + C_M \norm{\nabla y_u }_{L^{2q}(\Omega)}^2 \right].
    \end{align*}
    From this, \cref{ass:domain}, and the global boundedness of the gradient of solutions to Poisson's equation (see, e.g. \cite[Thm.~3.1, Rem.~3]{CianchiMazya2015}, \cite{Mazya2009,YangChen2018}), it follows that $\tilde{y} \in W^{1,\infty}(\Omega)$ and 
    \begin{equation*}
        \norm{\nabla \tilde{y}}_{L^{\infty}(\Omega)} \leq C\norm{\Delta \tilde{y}}_{L^q(\Omega)}.
    \end{equation*}
    It is therefore sufficient to prove that $\tilde{y} = y_u$. To this end, taking any $\phi \in H^1_0(\Omega)$ yields that $(b + a(y_u))\phi \in H^1_0(\Omega)$. Testing \eqref{eq:state-convex} by $(b + a(y_u))\phi$, a straightforward computation shows that
    \begin{equation*}
        \int_\Omega \left( \nabla \tilde{y} - \nabla y_u\right) \cdot \nabla \left(b + a(y_u)  \right) \phi + (b + a(y_u)) \nabla \phi \cdot \nabla \tilde{y} \dx = \int_\Omega u \phi \dx.
    \end{equation*}
    Now testing \eqref{eq:state-app} by $\phi$ and then subtracting the obtained identity from the above equality, we obtain
    \begin{equation*}
        \int_\Omega \left( \nabla \tilde{y} - \nabla y_u\right) \cdot \nabla \left[(b + a(y_u)) \phi \right] \dx = 0.
    \end{equation*}
    Choosing $\phi := \frac{1}{b + a(y_u)}(\tilde{y}- y_u) \in H^1_0(\Omega)$ then yields $\tilde{y} = y_u$.
\end{proof}

\section{Estimate of jump functional}\label{app:examples}

In this appendix, we verify the estimates for $\Sigma(y)$ in \cref{exam:Ey-ocsillation}.

We can assume without loss of generality that $E_{a} = \{ t \}$, i.e., $\barK= 1$ and $t_1 = t$. We first consider the right end point $\overline x_j$ of $[\underline{x}_j,\overline x_j]$. If $\overline x_j = \beta_0$, then the set $\{0 < | y - t | < r\} \cap (\overline x_j,\infty)$ is empty for any $r>0$, and thus
\begin{equation*}
    \lim_{r \to 0^+} \frac{1}{r} \int_\Omega \1_{\{0 < | y - t | < r\} \cap (\overline x_j,\infty)} \left| \frac{\partial y}{\partial x}\right| \dx = 0.
\end{equation*}
If $\overline x_j < \beta_0$, then there exists an $\bar \epsilon_j>0$ such that $y$ is monotone on $(\overline x_j, \overline x_j +\bar  \epsilon_j) \subset \Omega$ due the monotonicity assumption on $y$. Since $\overline x_j$ is the right end point of $[\underline{x}_j,\overline x_j]$, the sets $[\underline{x}_k,\overline x_k]$, $k \in J$, are mutually disjoint, and $y$ is continuous, we have $y(\overline x_j + \bar \epsilon_j) \neq y(\overline x_j) = t$.  
Setting $\bar r_j := |y(\overline x_j + \bar\epsilon_j) - y(\overline x_j)| > 0$, we derive from the monotonicity of $y$ on $(\overline x_j, \overline x_j +\bar \epsilon_j)$ that for all $r \in (0, \bar r_j)$,
\begin{equation*}
    \{0 < | y - t | < r\} \cap ( \overline x_j, \overline x_j + \bar \epsilon_j) = 
    \begin{cases}
        \left(\overline x_j, y^{-1}(t + r)\right) & \text{if $y$ is monotone increasing on $\left(\overline x_j, \overline x_j +\bar  \epsilon_j\right)$},\\
        \left(\overline x_j, y^{-1}(t - r)\right) & \text{if $y$ is monotone decreasing on $\left(\overline x_j, \overline x_j +\bar  \epsilon_j\right)$}.
    \end{cases}   
\end{equation*}
This implies that
\begin{equation*}
    \lim_{r \to 0^+}\frac{1}{r} \int_\Omega \1_{ \{ 0 < | y - t | < r\}\cap  ( \overline x_j, \overline x_j +\bar \epsilon_j)}\left| \frac{\partial y}{\partial x}\right| \dx  = 1.
\end{equation*}
Similarly, we have for the left end point $\underline{x}_j$ that there exists an $\underline \epsilon_j>0$ such that
\begin{equation*}
    \lim_{r \to 0^+}\frac{1}{r} \int_\Omega \1_{ \{ 0 < | y - t | < r\} \cap (  \underline x_j- \underline \epsilon_j, \underline x_j)} \left| \frac{\partial y}{\partial x}\right| \dx   = 
    \begin{cases}
        0 & \text{if } \underline{x}_j = \alpha_0,\\
        1 & \text{if } \underline{x}_j > \alpha_0.
    \end{cases}
\end{equation*}
Furthermore, if $\underline{x}_j > \alpha_0$, then $y$ is monotone on $(\underline x_j- \underline \epsilon_j, \underline x_j ) \subset \Omega$ and
\begin{equation*}
    \{0 < | y - t | < r\} \cap \left( \underline x_j - \underline \epsilon_j, \underline x_j\right) = 
    \begin{cases}
        \left(y^{-1}(t - r), \underline x_j\right) & \text{if $y$ is monotone increasing on $\left(\underline x_j - \underline  \epsilon_j, \underline x_j\right)$},\\
        \left(y^{-1}(t + r), \underline x_j\right) & \text{if $y$ is monotone decreasing on  $\left(\underline x_j - \underline  \epsilon_j, \underline x_j\right)$},
    \end{cases}   
\end{equation*}
for all $r \in (0, \underline r_j)$ with $\underline r_j := |y(\underline x_j - \underline \epsilon_j) - y(\underline x_j)| > 0$. Let us set $\epsilon_j := \min \{\underline \epsilon_j, \bar \epsilon_j \}$ and $r_j := \min \{\underline r_j, \bar r_j \}$. If $y$ is monotone increasing on $\left(\overline x_j, \overline x_j +\bar  \epsilon_j\right)$ and on $\left(\underline x_j - \underline  \epsilon_j, \underline x_j\right)$, then
\begin{equation*}
    \{0 < | y - t | < r\} \cap \left( \underline x_j - \epsilon_j, \overline x_j +\epsilon_j\right) = \left(\overline x_j, y^{-1}(t + r)\right) \cup \left(y^{-1}(t - r), \underline x_j\right)
\end{equation*}
for all $r \in (0, r_j)$.
Similar expressions hold for the other cases of monotonicity on each interval.
We hence obtain
\begin{equation*}
    \lim_{r \to 0^+}\frac{1}{r} \int_\Omega \1_{ \{ 0 < | y - t | < r\} \cap ( \underline x_j - \epsilon_j, \overline x_j +\epsilon_j)} \left| \frac{\partial y}{\partial x}\right| \dx  = 
    \begin{cases}
        0 & \text{if }\underline{x}_j, \overline x_j\in \{\alpha_0,\beta_0\},\\
        1 & \text{if }\underline{x}_j\in\{\alpha_0,\beta_0\}, \overline x_j\notin \{\alpha_0,\beta_0\} \text{ or vice versa},\\
        2 & \text{if }\underline{x}_j, \overline x_j\notin\{\alpha_0,\beta_0\}.
    \end{cases}
\end{equation*}
A standard argument then yields the desired conclusions.

\bibliographystyle{jnsao}
\bibliography{nsqlsoc}

\end{document}